\documentclass[11pt,a4paper]{article}
\usepackage{amssymb}
\usepackage{amsmath}
\usepackage{latexsym}
\usepackage{amsthm}
\usepackage{dsfont}
\usepackage{enumerate}
\usepackage{hyperref}

\newtheorem{thm}{Theorem}
\newtheorem{lem}[thm]{Lemma}
\newtheorem{cor}[thm]{Corollary}
\theoremstyle{definition}
\newtheorem{remark}{Remark}

\newcommand\extrafootertext[1]{
    \bgroup
    \renewcommand\thefootnote{\fnsymbol{footnote}}
    \renewcommand\thempfootnote{\fnsymbol{mpfootnote}}
    \footnotetext[0]{#1}
    \egroup
}

\usepackage{xpatch}
\xpatchcmd{\proof}{\itshape}{\normalfont\proofnameformat}{}{}
\newcommand{\proofnameformat}{}

\begin{document}

\renewcommand{\proofnameformat}{\bfseries}

\begin{center}
{\Large\textbf{Equidistribution of number theoretic point sets in the torus and the sphere in Wasserstein metric and $L^2$ discrepancy}}

\vspace{10mm}

\textbf{Bence Borda and Filippo Giannoni}

\let\thefootnote\relax

\let\thefootnote\relax
\end{center}

\vspace{10mm}

\begin{abstract}
We find the exact rate of equidistribution of several arithmetic point configurations in Wasserstein metric and $L^2$ discrepancy with respect to balls. Kronecker sequences in the torus and the Fibonacci lattice on the Euclidean unit sphere are shown to be more evenly distributed in Wasserstein metric than random points. Modular hyperbolas and modular parabolas in the torus behave like random points, although the constants in the asymptotics of the $L^2$ discrepancy with respect to balls differ from the random case. We also prove a sharp inequality between the quadratic Wasserstein metric and the $L^2$ discrepancy with respect to balls for an arbitrary probability measure on the torus or the sphere. In particular, we show that any point set that is optimally close to the uniform measure in $L^2$ discrepancy with respect to balls is also optimal in the quadratic Wasserstein metric.
\end{abstract}

\section{Introduction}

The Wasserstein metric $W_p$, $1 \le p \le \infty$, is a metric on the set of Borel probability measures $\mathcal{P}(M)$ on a metric space $M$. The notion originates in the theory of optimal transport and plays an important role in partial differential equations, geometric analysis and probability theory \cite{VI1,VI2} with further applications in computer vision \cite{RTG} and machine learning \cite{ACB}. Applications of the Wasserstein metric as a way of quantifying equidistribution results in number theory have also been explored in recent years, including in the context of Kronecker sequences and grids \cite{BS,ST}, quadratic residues \cite{BS,GR,ST}, exponential sums in finite fields \cite{KU}, prime number races \cite{BHU}, Heegner points and closed geodesics on modular surfaces \cite{HU2} and arithmetic quantum unique ergodicity \cite{GA,HU2}.

The main goal of the present paper is to find the exact rate of equidistribution of several number theoretic point sets and sequences in Wasserstein metric. We study Kronecker sequences, lattices, modular hyperbolas and modular parabolas in the torus $\mathbb{R}^d /\mathbb{Z}^d$, and the Fibonacci lattice on the Euclidean unit sphere $\mathbb{S}^2$. We also compare the behavior of the quadratic Wasserstein metric to the $L^2$ discrepancy with respect to balls. All formal definitions and notation are introduced in Section \ref{notation_section}.

The rate of equidistribution of a finite point set $a_1, a_2, \ldots, a_N \in M$ is quantified by the distance from the empirical measure $\mu_N=N^{-1} \sum_{n=1}^N \delta_{a_n}$, where $\delta_a$ is the Dirac measure concentrated at the point $a \in M$, to the uniform probability measure $\mathrm{Vol}$ on $M$. The best possible rate in both the torus $M=\mathbb{R}^d / \mathbb{Z}^d$ and the sphere $M=\mathbb{S}^d$ is
\begin{equation}\label{Wpbestrate}
\inf_{a_1, a_2, \ldots, a_N \in M} W_p (\mu_N, \mathrm{Vol}) \asymp N^{-1/d}, \qquad 1 \le p \le \infty,
\end{equation}
where the infimum is over all point sets of size $N$. Allowing arbitrary weights in the empirical measure, we have more precisely
\begin{equation}\label{Wpweightedbestrate}
\inf_{\substack{\mu \in \mathcal{P}(M) \\ |\mathrm{supp}(\mu)| \le N}} W_p (\mu, \mathrm{Vol}) = (c(p,d)+o(1)) N^{-1/d} \quad \textrm{as } N \to \infty,
\end{equation}
where the infimum is over all probability measures $\mu=\sum_{n=1}^N w_n \delta_{a_n}$ supported on at most $N$ points with weights $w_n \ge 0$, $\sum_{n=1}^N w_n=1$. The exact value of the constant $c(p,d)$ is known only in some special cases \cite{GL}. In comparison, for a sample of independent, identically distributed (i.i.d.) random points $a_1, a_2, \ldots, a_N \in M$, each uniformly distributed in $M$, and for all $1 \le p < \infty$, the corresponding rate in expectation is
\begin{equation}\label{iidrate}
\mathbb{E} (W_p (\mu_N, \mathrm{Vol})) \asymp \left\{ \begin{array}{ll} N^{-1/2} & \textrm{if } d=1, \\ (\log N)^{1/2} N^{-1/2} & \textrm{if } d=2, \\ N^{-1/d} & \textrm{if } d \ge 3. \end{array} \right.
\end{equation}
Thus i.i.d.\ random points attain the best possible rate in dimension $d \ge 3$, but not in dimensions $d=1$ and 2.

Relations \eqref{Wpbestrate} and \eqref{Wpweightedbestrate} actually hold in all $d$-dimensional compact, connected Riemannian manifolds without boundary \cite{GL,KL}. The upper bounds in \eqref{iidrate} in this generality were first proved by Ledoux \cite{LE}. The lower bound in \eqref{iidrate} in dimension $d \ge 3$ follows from \eqref{Wpbestrate}, whereas the lower bound in \eqref{iidrate} in dimension $d=2$ is due to Ajtai, Koml\'os and Tusn\'ady \cite{AKT} in the unit square $[0,1]^2$, see also \cite[Theorem 5.2]{AST} for the torus $\mathbb{R}^2 / \mathbb{Z}^2$ and \cite[Proposition 8]{HPZ} for the sphere $\mathbb{S}^2$.

The most classical example of an equidistributed sequence in number theory is the Kronecker sequence $n \alpha + \mathbb{Z}^d$ in the torus $\mathbb{R}^d/\mathbb{Z}^d$ with a given vector $\alpha \in \mathbb{R}^d$. Brown and Steinerberger \cite[Theorem 5]{BS} proved that the Kronecker sequence attains the best possible rate in the quadratic Wasserstein metric $W_2(\mu_N, \mathrm{Vol}) \ll N^{-1/d}$ in dimension $d \ge 2$ provided the vector $\alpha$ is badly approximable, that is, $\| \langle k, \alpha \rangle \|_{\mathbb{R}/\mathbb{Z}} \gg \| k \|_{\infty}^{-d}$ for all $k \in \mathbb{Z}^d$, $k \neq 0$. We generalize their estimate to the Wasserstein metric of arbitrary order $p$.
\begin{thm}\label{kroneckerbadlyapprox_theorem} Let $d \ge 2$, let $\alpha \in \mathbb{R}^d$ be a badly approximable vector and let $\mu_N=N^{-1} \sum_{n=1}^N \delta_{n\alpha + \mathbb{Z}^d}$. We have
\[ W_p (\mu_N, \mathrm{Vol}) \ll p^{(d+1)/(2d)} N^{-1/d}, \quad 1 \le p < \infty \]
and
\[ W_{\infty} (\mu_N, \mathrm{Vol}) \ll (\log N)^{(d+1)/(2d)} N^{-1/d} \]
with implied constants that depend only on $\alpha$.
\end{thm}
Kronecker sequences thus attain the best possible rate \eqref{Wpbestrate} for all finite $p$, and in particular, are more evenly distributed than random points in dimension $d=2$. We do not know whether the logarithmic factor is necessary in the estimate for $W_{\infty}$. In comparison, in dimension $d=1$ the rates are $W_p(\mu_N, \mathrm{Vol}) \ll (\log N)^{1/2} N^{-1}$, $1 \le p < \infty$, and $W_{\infty}(\mu_N, \mathrm{Vol}) \ll (\log N) N^{-1}$, and both estimates are sharp for infinitely many $N$ \cite{GR}. Generalizations of Theorem \ref{kroneckerbadlyapprox_theorem} to a.e.\ $\alpha$ and to higher rank lattices are discussed in Section \ref{kronecker_section}.

We also consider a lattice-like point set on the 2-dimensional unit sphere obtained by mapping a lattice in the unit square to the sphere using a projection that originates in cartography. The Lambert cylindrical equal area projection is defined as the map $\phi: [0,1]^2 \to \mathbb{S}^2$,
\[ \phi (x_1, x_2) = \left( 2 \sqrt{x_2-x_2^2} \cos (2 \pi x_1), 2 \sqrt{x_2-x_2^2} \sin (2 \pi x_1), 1-2x_2 \right) . \]
Geometrically, $\phi$ stretches the unit square to the rectangle $[0,2\pi] \times [-1, 1]$, identifies the two vertical sides by folding the rectangle to a cylinder of height 2 that touches the unit sphere along the equator, then projects the points on the cylinder to the sphere orthogonally to the vertical axis of the sphere. Most importantly, $\phi$ preserves the normalized surface area.

Given $\alpha \in \mathbb{R}$ and $N \in \mathbb{N}$, the lattice $\{ (\{ n \alpha \}, n/N) : 0 \le n < N \} \subset [0,1]^2$, where $\{ \cdot \}$ denotes the fractional part function, is mapped to the set
\[ S_N (\alpha) = \left\{ \phi \left( n \alpha, \frac{n}{N} \right) : 0 \le n < N \right\} \subset \mathbb{S}^2 . \]
In the special case when $N=F_m$ is a Fibonacci number and $\alpha=F_{m-1}/F_m$, this is known as the Fibonacci lattice on the sphere. An irrational version of the same set with the golden mean $\alpha=(\sqrt{5}-1)/2$ is defined for all $N \in \mathbb{N}$. Numerical evidence suggests that the $L^{\infty}$ discrepancy with respect to balls and the Riesz energy of the Fibonacci lattice are strictly smaller than those of random points, but a formal proof of this remains wide open \cite{ABD,HMS}. In this paper, we prove that the Fibonacci lattice is indeed more evenly distributed than random points when measured in the Wasserstein metric of order $1 \le p < 2$. Our result applies to all badly approximable $\alpha$ such as the golden mean in the irrational setting, and also to best rational approximations to badly approximable numbers such as $F_{m-1}/F_m$ in the rational setting.

\begin{thm}\label{fibonacci_theorem} Let $\alpha \in \mathbb{R}$, $N \in \mathbb{N}$, and assume $\min_{1 \le n \le \sqrt{N}} n \| n \alpha \|_{\mathbb{R}/\mathbb{Z}} \ge c$ with some constant $c>0$. Let $\mu_N=N^{-1} \sum_{s \in S_N(\alpha)} \delta_s \in \mathcal{P}(\mathbb{S}^2)$ denote the empirical measure of the set $S_N(\alpha)$. For any $1 \le p <2$, we have
\[ W_p (\mu_N, \mathrm{Vol}) \ll (2-p)^{-5/2} N^{-1/2} \]
with an implied constant that depends only on $c$.
\end{thm}

The Fibonacci lattice can be used as a quadrature rule for the purposes of numerical integration on the sphere. By Kantorovich duality \eqref{kantorovich}, Theorem \ref{fibonacci_theorem} with $p=1$ states that the worst case integration error in the class of Lipschitz functions attains the best possible rate $N^{-1/2}$ among all points sets of size $N$. In particular, this is better than the rate $(\log N)^{1/2} N^{-1/2}$ attained by random points.

The numerical evidence in \cite[Figure 26]{HMS} together with Theorem \ref{W2D2_theorem} below suggests that the factor $(2-p)^{-5/2}$ can be removed from the estimate in Theorem \ref{fibonacci_theorem} and in particular, $W_2 (\mu_N, \mathrm{Vol}) \ll N^{-1/2}$, but this remains open.

The proofs of Theorems \ref{kroneckerbadlyapprox_theorem} and \ref{fibonacci_theorem} are based on smoothing inequalities due to the first author and Cuenin \cite{BC} that estimate $W_p(\mu, \mathrm{Vol})$, $1 \le p \le \infty$, in terms of the Fourier transform of the probability measure $\mu$. In the case of an empirical measure $\mu_N$, the Fourier transform becomes an exponential sum. In addition to the lattices in Theorems \ref{kroneckerbadlyapprox_theorem} and \ref{fibonacci_theorem}, there is a large number of point sets and sequences originating in number theory for which the corresponding exponential sums can be effectively estimated. As an illustration of the flexibility of our approach, in Sections \ref{modularhyperbola_section} and \ref{modularparabola_section} we work out the details for modular hyperbolas and modular parabolas in the torus, the Fourier transforms of which are Kloosterman sums and quadratic Gauss sums, respectively.

Finally, we clarify the relationship between the quadratic Wasserstein metric $W_2$ and the $L^2$ discrepancy with respect to balls, denoted by $D_2$, another classical way of quantifying equidistribution. The distance from an arbitrary probability measure $\mu$ to the uniform probability measure $\mathrm{Vol}$ in these two metrics satisfy a general inequality, which we believe to be of independent interest.
\begin{thm}\label{W2D2_theorem} Let $M=\mathbb{R}^d/\mathbb{Z}^d$ or $\mathbb{S}^d$. For any $\mu \in \mathcal{P}(M)$, we have $W_2(\mu, \mathrm{Vol}) \ll D_2 (\mu, \mathrm{Vol})^{2/(d+1)}$ with an implied constant that depends only on $d$. The exponent $2/(d+1)$ is best possible.
\end{thm}

In the $1$-dimensional torus $\mathbb{R}/\mathbb{Z}$, we actually have the identity $W_2(\mu, \mathrm{Vol})=2^{-1/2} D_2 (\mu, \mathrm{Vol})$ for all $\mu \in \mathcal{P}(\mathbb{R}/\mathbb{Z})$ \cite{GR}. No such identity holds in dimension $d \ge 2$.

The best possible rate in the $D_2$ metric in both the torus $M=\mathbb{R}^d / \mathbb{Z}^d$ and the sphere $M=\mathbb{S}^d$ is \cite{BE}
\begin{equation}\label{D2bestrate}
\inf_{a_1, a_2, \ldots, a_N \in M} D_2 (\mu_N, \mathrm{Vol}) \asymp N^{-(d+1)/(2d)} .
\end{equation}
Theorem \ref{W2D2_theorem} thus shows that if a finite point set $a_1, a_2, \ldots, a_N \in M$ has optimally small $L^2$ discrepancy with respect to balls $D_2 (\mu_N, \mathrm{Vol}) \ll N^{-(d+1)/(2d)}$, then it is also optimal in the quadratic Wasserstein metric $W_2 (\mu_N, \mathrm{Vol}) \ll N^{-1/d}$. This observation proves that the exponent $2/(d+1)$ in Theorem \ref{W2D2_theorem} is best possible.

As $\inf_{|\mathrm{supp}(\mu)| \le N} W_1 (\mu, \mathrm{Vol}) \gg N^{-1/d}$ is straightforward to prove using Kantorovich duality \eqref{kantorovich}, Theorem \ref{W2D2_theorem} also provides a new proof of the general Roth-type lower bound
\[ \inf_{\substack{\mu \in \mathcal{P}(M) \\ |\mathrm{supp} (\mu)| \le N}} D_2 (\mu, \mathrm{Vol}) \gg N^{-(d+1)/(2d)} . \]
In fact, this is slightly more general than the usual formulation \eqref{D2bestrate}, which applies only to empirical measures with equal weights.

The rest of the paper is organized as follows. Section \ref{results_section} gives all formal definitions and states further results. Section \ref{approach_section} is an overview of our general approach based on harmonic analysis. Section \ref{wassersteinL2discrepancy_section} contains the proof of Theorem \ref{W2D2_theorem}, Section \ref{torus_section} contains the proof of all main results in the torus, and Section \ref{sphere_section} contains the proof of Theorem \ref{fibonacci_theorem}.

\section{Notation and further results}\label{results_section}

\subsection{Notation}\label{notation_section}

Throughout, $\langle x,y \rangle=\sum_{i=1}^d x_i y_i$ and $|x|=(\sum_{i=1}^d x_i^2)^{1/2}$ denote the standard Euclidean scalar product and norm on $\mathbb{R}^d$. The notations $a_N \ll b_N$ and $a_N=O(b_N)$ both mean that there exists a constant $C>0$ such that $|a_N| \le C b_N$ for all $N \ge 1$. We write $a_N \asymp b_N$ if $a_N \ll b_N \ll a_N$.

Let $(M,\rho_M)$ be a compact metric space and let $\mathcal{P}(M)$ denote the set of Borel probability measures on $M$. The Wasserstein metric of order $1 \le p < \infty$ is defined as
\[ W_p (\mu, \nu) = \inf_{\pi \in \mathrm{Coupling}(\mu, \nu)} \left( \int_{M \times M} \rho_M (x,y)^p \, \mathrm{d}\pi (x,y) \right)^{1/p}, \qquad \mu, \nu \in \mathcal{P}(M), \]
where $\mathrm{Coupling}(\mu, \nu)$ is the set of $\pi \in \mathcal{P}(M \times M)$ with marginals $\pi(A \times M)=\mu(A)$ and $\pi(M \times A)=\nu (A)$, $A \subseteq M$ Borel. The Wasserstein metric of order $p=\infty$ is defined in terms of the $L^{\infty}(\pi)$ norm as
\[ W_{\infty}(\mu, \nu) = \inf_{\pi \in \mathrm{Coupling}(\mu, \nu)} \inf \{ K \ge 0 : \rho_M (x,y) \le K \,\, \pi-\textrm{a.e.} \} . \]
Given any $\mu, \nu \in \mathcal{P}(M)$, we have $W_p (\mu, \nu) \le W_q (\mu, \nu)$ for $1 \le p \le q \le \infty$ and $\lim_{p \to \infty}W_p(\mu, \nu)=W_{\infty}(\mu, \nu)$. In the special case $p=1$, the Kantorovich duality formula states
\begin{equation}\label{kantorovich}
W_1 (\mu, \nu) = \sup_{f \,\, 1\textrm{-Lipschitz}} \left| \int_M f \, \mathrm{d} \mu - \int_M f \, \mathrm{d} \nu \right|,
\end{equation}
where the supremum is over all $f: M \to \mathbb{R}$ such that $|f(x)-f(y)| \le \rho_M (x,y)$ for all $x,y \in M$. We refer to Villani \cite{VI1,VI2} for a comprehensive treatment of optimal transport and the Wasserstein metric.

In this paper, we work with two metric spaces: the flat torus $\mathbb{R}^d/\mathbb{Z}^d$ equipped with the metric $\rho_{\mathbb{R}^d/\mathbb{Z}^d}(x,y)=\| x-y \|_{\mathbb{R}^d/\mathbb{Z}^d}$, where $\| x \|_{\mathbb{R}^d/\mathbb{Z}^d} = \min_{k \in \mathbb{Z}^d} |x-k|$ denotes the Euclidean distance to the nearest lattice point, and the Euclidean unit sphere $\mathbb{S}^d = \{ x \in \mathbb{R}^{d+1} : |x|=1 \}$ equipped with the geodesic metric $\rho_{\mathbb{S}^d}(x,y)=\arccos (\langle x,y \rangle)$. In both spaces, $\mathrm{Vol}$ denotes the uniform probability measure. That is, $\mathrm{Vol}$ is the Lebesgue measure on $\mathbb{R}^d/\mathbb{Z}^d$ identified with the unit cube $[0,1)^d$, whereas $\mathrm{Vol}$ is the normalized surface area measure on $\mathbb{S}^d$.

Let $B(x,r)=\{ y \in M : \rho_M (x,y)<r \}$ denote the open ball with center $x$ and radius $r$. The $L^p$ discrepancy with respect to balls of fixed radius is defined in the torus as
\[ D_p (\mu, \nu, r) = \left( \int_{\mathbb{R}^d / \mathbb{Z}^d} |\mu(B(x,r)) - \nu (B(x,r))|^p \, \mathrm{d} \mathrm{Vol}(x) \right)^{1/p}, \quad \mu, \nu \in \mathcal{P}(\mathbb{R}^d / \mathbb{Z}^d), \,\, 0<r \le \frac{1}{2} . \]
The $L^p$ discrepancy with respect to balls in the torus is defined via a further averaging over all $r$ up to the injectivity radius $1/2$ as
\[ D_p (\mu, \nu) = \left( 2 \int_0^{1/2} \int_{\mathbb{R}^d / \mathbb{Z}^d} |\mu(B(x,r)) - \nu (B(x,r))|^p \, \mathrm{d} \mathrm{Vol}(x) \, \mathrm{d}r \right)^{1/p}, \quad \mu, \nu \in \mathcal{P}(\mathbb{R}^d / \mathbb{Z}^d). \]
The $L^p$ discrepancy with respect to balls on the sphere is defined similarly as
\[ \begin{split} D_p (\mu, \nu) &= \left( \frac{1}{2} \int_0^{\pi} \int_{\mathbb{S}^d} |\mu(B(x,r)) - \nu(B(x,r))|^p \sin r \, \mathrm{d}\mathrm{Vol}(x) \, \mathrm{d}r \right)^{1/p} \\ &= \left( \frac{1}{2} \int_{-1}^1 \int_{\mathbb{S}^d} |\mu(C(x,t)) - \nu(C(x,t))|^p \, \mathrm{d}\mathrm{Vol}(x) \, \mathrm{d}t \right)^{1/p}, \quad \mu, \nu \in \mathcal{P}(\mathbb{S}^d) , \end{split} \]
where $C(x,t) = \{ y \in \mathbb{S}^d : \langle x, y \rangle > t \}=B(x,\arccos t)$, $-1 \le t \le 1$, denotes the spherical cap centered at $x$ of height $1-t$. The $L^{\infty}$ discrepancy with respect to balls on both spaces is defined as
\[ D_{\infty} (\mu, \nu) =  \sup_{0<r \le r_0} \sup_{x \in M} |\mu (B(x,r)) - \nu (B(x,r))|, \]
where $r_0=1/2$ resp.\ $\pi$ in the torus resp.\ sphere. Clearly, $D_p (\mu, \nu) \le D_q (\mu, \nu)$ for all $1 \le p \le q \le \infty$.

\subsection{Kronecker sequences and higher rank lattices}\label{kronecker_section}

Given a vector $\alpha \in \mathbb{R}^d$, the Kronecker sequence is defined as $n \alpha + \mathbb{Z}^d$ in the torus $\mathbb{R}^d/\mathbb{Z}^d$. The rate of convergence of the empirical measure of the Kronecker sequence to the Lebesgue measure $\mathrm{Vol}$ on the torus is sensitive to the Diophantine approximation properties of $\alpha$. We work under the assumption
\[ I = \inf_{\substack{k \in \mathbb{Z}^d \\ k \neq 0}} \left( \psi (\| k \|_{\infty}) \| \langle k, \alpha \rangle \|_{\mathbb{R}/\mathbb{Z}} \right) >0, \]
where $\psi: [1,\infty) \to [1,\infty)$ is a nondecreasing function and $\| k \|_{\infty}=\max_{1 \le i \le d} |k_i|$.

More generally, let $\alpha_1, \alpha_2, \ldots, \alpha_m \in \mathbb{R}^d$ be a system of vectors and consider the rank $m$ lattice
\[ \left\{ n_1 \alpha_1 + n_2 \alpha_2 + \cdots + n_m \alpha_m + \mathbb{Z}^d : n_1, n_2, \ldots, n_m \in \mathbb{Z} \right\} \subset \mathbb{R}^d / \mathbb{Z}^d. \]
We work under the multiplicative assumption
\begin{equation}\label{diophantinelattice}
I = \inf_{\substack{k \in \mathbb{Z}^d \\ k \neq 0}} \bigg( \psi (\| k \|_{\infty}) \prod_{j=1}^m \| \langle k, \alpha_j \rangle \|_{\mathbb{R}/\mathbb{Z}} \bigg) >0.
\end{equation}
Let $F^{-1}(y)=\sup \{ x \ge 1 : F(x) \le y \}$ denote the generalized inverse function of a nondecreasing function $F$.

\begin{thm}\label{lattice_theorem} Let $d \ge 2$, $m \ge 1$ and assume the system $\alpha_1, \alpha_2, \ldots, \alpha_m \in \mathbb{R}^d$ satisfies \eqref{diophantinelattice} with a function $\psi: [1, \infty) \to [1, \infty)$ such that $x^{-d} \psi (x)$ is nondecreasing and there exists a constant $c>0$ for which $\psi(xy) \le c e^{x^2} \psi(y)$ holds for all $x,y \ge 1$. Let $N_1, N_2, \ldots, N_m \ge 2$ be integers, let $N=\prod_{j=1}^m N_j$ and let
\[ \mu_N= \frac{1}{N} \sum_{n_1=1}^{N_1} \sum_{n_2=1}^{N_2} \cdots \sum_{n_m=1}^{N_m} \delta_{n_1 \alpha_1+ n_2 \alpha_2 + \cdots + n_m \alpha_m+\mathbb{Z}^d}. \]
Let $F_p: [1,\infty) \to [1,\infty)$, $F_p(x)=(\psi(x)+p x^d) (\log (\psi(x)+2))^{(m-1)(p-1)/p}$. For any $2 \le p < \infty$, we have
\[ W_p (\mu_N, \mathrm{Vol}) \ll \frac{p^{1/2}}{F_p^{-1} \left( \frac{p^{1/2} N}{\min \left\{ p, \log \left( \frac{N}{\psi^{-1}(N)^d}+2 \right) \right\}} \right)} \]
and
\[ W_{\infty} (\mu_N, \mathrm{Vol}) \ll \frac{(\log N)^{1/2}}{F_{\log N}^{-1} \left( \frac{(\log N)^{1/2} N}{\min \left\{ \log N, \log \left( \frac{N}{\psi^{-1}(N)^d}+2 \right) \right\}} \right)} \]
with implied constants that depend only on $I$, $c$, $d$ and $m$.
\end{thm}

Dirichlet's approximation theorem \cite[p.\ 13]{CA} immediately shows that if $\psi(x)=o(x^d)$, then there does not exist a system of vectors $\alpha_1, \alpha_2, \ldots, \alpha_m \in \mathbb{R}^d$ that satisfies the multiplicative assumption \eqref{diophantinelattice}. Let us call a system of vectors $\alpha_1, \alpha_2, \ldots, \alpha_m \in \mathbb{R}^d$ badly approximable in the sense of \eqref{diophantinelattice} if they satisfy \eqref{diophantinelattice} with $\psi(x)=x^d$. If $m=1$, then this is the usual notion of badly approximable vectors $\alpha \in \mathbb{R}^d$, which are known to exist for any $d \ge 1$ \cite[p.\ 15]{CA}. In this case $F_p(x) \asymp p x^d$, $F_p^{-1}(y) \asymp p ^{-1/d} y^{1/d}$ and Theorem \ref{lattice_theorem} yields $W_p(\mu_N, \mathrm{Vol}) \ll p^{(d+1)/(2d)} N^{-1/d}$ and $W_{\infty}(\mu_N, \mathrm{Vol}) \ll (\log N)^{(d+1)/(2d)} N^{-1/d}$. This proves Theorem \ref{kroneckerbadlyapprox_theorem}. In contrast, whether a badly approximable system of vectors in the sense of \eqref{diophantinelattice} with $m \ge 2$ exists is an open problem; in fact, Littlewood's conjecture states that such a system does not exist for $d=1$, $m=2$. Note that in the case $d, m \ge 2$, badly approximable in the sense of \eqref{diophantinelattice} is somewhere between the usual notion of a badly approximable system, and the notion of a multiplicatively badly approximable system as introduced by Bugeaud \cite{BU}.

Considering a system of vectors $\alpha_1, \alpha_2, \ldots, \alpha_m \in \mathbb{R}^d$ as an element of $\mathbb{R}^{dm}$, the corresponding metric result states that a.e.\ system satisfies \eqref{diophantinelattice} with $\psi(x)=x^d (\log x)^m (\log \log x)^{1+\varepsilon}$, $\varepsilon >0$. Since the condition \eqref{diophantinelattice} is not classical, we give a short proof of this fact in Lemma \ref{metriclemma} below. In this case $F_p^{-1}(y) \asymp y^{1/d} (\log y)^{-m/d-(m-1)(p-1)/(pd)} (\log \log y)^{-(1+\varepsilon)/d}$ for $2 \le p \ll \log y$, and we immediately deduce the following.

\begin{cor}\label{lattice_corollary} Let $d \ge 2$, $m \ge 1$, $2 \le p < \infty$ and $\varepsilon>0$. Let $N_1, N_2, \ldots, N_m \ge 2$ be integers, let $N=N_1 N_2 \cdots N_m$ and let
\[ \mu_N= \frac{1}{N} \sum_{n_1=1}^{N_1} \sum_{n_2=1}^{N_2} \cdots \sum_{n_m=1}^{N_m} \delta_{n_1 \alpha_1+ n_2 \alpha_2 + \cdots + n_m \alpha_m+\mathbb{Z}^d}. \]
For a.e.\ $\alpha_1, \alpha_2, \ldots, \alpha_m \in \mathbb{R}^d$, we have
\[ \begin{split} W_p (\mu_N, \mathrm{Vol}) &\ll (\log N)^{m/d+(m-1)(p-1)/(pd)} (\log \log N)^{1/d+\varepsilon} N^{-1/d}, \\ W_{\infty} (\mu_N, \mathrm{Vol}) &\ll (\log N)^{(d+4m-3)/(2d)} (\log \log N)^{2/d+\varepsilon} N^{-1/d} \end{split} \]
with implied constants that depend only on $\alpha_1, \alpha_2, \ldots, \alpha_m$, $d$, $m$, $p$ and $\varepsilon$.
\end{cor}

The estimates in Corollary \ref{lattice_corollary} are within logarithmic factors from the best possible rate $N^{-1/d}$ in \eqref{Wpbestrate}. In comparison, the best known estimate for the $L^{\infty}$ discrepancy with respect to balls of the Kronecker sequence $n \alpha + \mathbb{Z}^d$ is $D_{\infty}(\mu_N, \mathrm{Vol}) \ll (\log N)^{2+\varepsilon} N^{-2/(d+1)}$ for a.e.\ $\alpha \in \mathbb{R}^d$, but it is not known whether the exponent $-2/(d+1)$ is optimal \cite{HA}.

\subsection{Modular hyperbolas}\label{modularhyperbola_section}

Given a positive integer $q$ and an integer $a$ coprime with $q$, consider the set
\[ S_{a,q,d} = \left\{ \left( \frac{x_1}{q}, \frac{x_2}{q}, \ldots, \frac{x_d}{q} \right) + \mathbb{Z}^d : x_1, x_2, \ldots, x_d \in (\mathbb{Z}/q\mathbb{Z})^{\times}, x_1 x_2 \cdots x_d \equiv a \pmod{q} \right\} \subset \mathbb{R}^d / \mathbb{Z}^d \]
of cardinality $N=\varphi(q)^{d-1}$. In dimension $d=2$, the set
\[ S_{a,q,2}= \left\{ \left( \frac{x}{q}, \frac{a\overline{x}}{q} \right) + \mathbb{Z}^2 : x \in (\mathbb{Z}/q\mathbb{Z})^{\times} \right\} \subset \mathbb{R}^2 / \mathbb{Z}^2 \]
is known as the modular hyperbola. Here $\varphi(q)$ is the Euler totient function, $(\mathbb{Z}/q\mathbb{Z})^{\times}$ is the set of modulo $q$ remainders that are coprime with $q$ and $\overline{x}$ denotes the multiplicative inverse of $x$ modulo $q$. We refer to Shparlinski \cite{SH} for a general survey and to \cite{BRS,HU2} for a detailed study of equidistribution properties of modular hyperbolas in dimension $d=2$.

The Fourier transform of the empirical measure of the set $S_{a,q,d}$ is a Kloosterman sum. A smoothing inequality discussed in Section \ref{smoothing_section} together with the Weil bound for Kloosterman sums will immediately show that the same rate is attained in the quadratic Wasserstein metric as for i.i.d.\ random points \eqref{iidrate}. As for the $L^2$ discrepancy with respect to balls, the same rate as i.i.d.\ random points is attained only in dimensions $d=2$ and $3$, but not in dimension $d \ge 4$. For the sake of simplicity, we only consider prime $q$.

\begin{thm}\label{modularhyperbola_theorem} Let $q$ be a prime, let $a \in (\mathbb{Z}/q\mathbb{Z})^{\times}$ and let $\mu_{a,q,d}= (q-1)^{-(d-1)} \sum_{s \in S_{a,q,d}} \delta_s$ denote the empirical measure of the set $S_{a,q,d}$. If $d=2$, then
\begin{equation}\label{W2estimated=2}
W_2 (\mu_{a,q,2}, \mathrm{Vol}) \ll (\log q)^{1/2} q^{-1/2}
\end{equation}
with a universal implied constant. If $d \ge 3$, then
\[ W_2(\mu_{a,q,d}, \mathrm{Vol}) \asymp q^{-(d-1)/d} \quad \textrm{and} \quad D_2 (\mu_{a,q,d}, \mathrm{Vol}) \asymp q^{-1} \]
with implied constants that depend only on $d$.
\end{thm}

It would be interesting to see whether the estimate \eqref{W2estimated=2} can be refined to an asymptotic relation $W_2 (\mu_{a,q,2}, \mathrm{Vol}) =(c+o(1)) (\log q)^{1/2} q^{-1/2}$ with some constant $c$. While this remains open, we show that \eqref{W2estimated=2} is sharp on average.

\begin{thm}\label{modularhyperbolaW2_theorem} Let $q$ be a prime. Then
\[ \frac{1}{q-1} \sum_{a \in (\mathbb{Z}/q\mathbb{Z})^{\times}} W_2 (\mu_{a,q,2}, \mathrm{Vol})^2 \asymp \frac{\log q}{q} \]
with a universal implied constant.
\end{thm}

Humphries \cite[Remark 3.2]{HU1} proved that in dimension $d=2$ and for prime $q$, the $L^2$ discrepancy with respect to balls with fixed radius $0<r \le 1/2$ satisfies
\[ D_2 (\mu_{a,q,2}, \mathrm{Vol}, r) \ll r q^{-1/2} . \]
Consequently, after further averaging over $r$,
\[ D_2 (\mu_{a,q,2}, \mathrm{Vol}) \ll q^{-1/2} . \]
The former relation was also shown to be sharp in the range $r \ll (\log q)^{-1/2} q^{-1/2}$ \cite[Theorem 1.4]{HU1}. We prove that both relations are sharp on average, the former in the full range $0<r\le 1/2$.
\begin{thm}\label{modularhyperbolaD2_theorem} Let $q$ be a prime. For all $0<r\le 1/2$, we have
\[ \frac{1}{q-1} \sum_{a \in (\mathbb{Z}/q\mathbb{Z})^{\times}} D_2(\mu_{a,q,2}, \mathrm{Vol},r)^2 = \frac{\pi r^2 - (32/3) r^3 + \pi^2 r^4}{q-1} + \frac{\pi r^2 - \pi^2 r^4}{(q-1)^2} + O \left( \min \left\{ \frac{r^2}{q^2}, \frac{r}{q^3} \right\} \right) \]
and
\[ \frac{1}{q-1} \sum_{a \in (\mathbb{Z}/q\mathbb{Z})^{\times}} D_2(\mu_{a,q,2}, \mathrm{Vol})^2 = \frac{\pi/12-1/3+\pi^2/80}{q-1} + \frac{\pi/12-\pi^2/80}{(q-1)^2} + O \left( \frac{1}{q^3} \right) \]
with universal implied constants.
\end{thm}

The orders of magnitude in Theorem \ref{modularhyperbolaD2_theorem} are in perfect agreement with the case of the empirical measure $\mu_N$ of $N$ i.i.d.\ random points in $\mathbb{R}^2 / \mathbb{Z}^2$, for which
\begin{equation}\label{iidD2rate}
\mathbb{E} (D_2(\mu_N, \mathrm{Vol}, r)^2) = \frac{\pi r^2 -\pi^2 r^4}{N} \quad \textrm{and} \quad \mathbb{E} (D_2(\mu_N, \mathrm{Vol})^2) = \frac{\pi / 12 - \pi^2 / 80}{N}.
\end{equation}
The proof of \eqref{iidD2rate} is recalled in Section \ref{D2harmonic_section}. Interestingly, the constants in Theorem \ref{modularhyperbolaD2_theorem} deviate from the random case.

The average of the $L^2$ discrepancy over primes $q$ instead of over residue classes $a \in (\mathbb{Z}/q\mathbb{Z})^{\times}$ is closely related to deep open problems about the distribution of Kloosterman sum angles. We will show that if the horizontal Sato--Tate conjecture for Kloosterman sums (see \eqref{satotate}) is true, then for an arbitrary nonzero integer $a$,
\begin{equation}\label{averageD2overprimes}
\begin{split} \lim_{Q \to \infty} \frac{1}{\pi(Q)} \sum_{\substack{q \le Q \\ q \textrm{ prime}, \,\, q \nmid a}} (q-1) D_2(\mu_{a, q, 2}, \mathrm{Vol}, r)^2 &= \pi r^2 - \frac{32}{3} r^3 + \pi^2 r^4 \quad \textrm{uniformly in } 0<r \le \frac{1}{2}, \\ \lim_{Q \to \infty} \frac{1}{\pi(Q)} \sum_{\substack{q \le Q \\ q \textrm{ prime}, \,\, q \nmid a}} (q-1) D_2(\mu_{a, q, 2}, \mathrm{Vol})^2 &= \frac{\pi}{12} - \frac{1}{3} + \frac{\pi^2}{80} . \end{split}
\end{equation}
Here $\pi(Q)$ denotes the prime-counting function.

In comparison, the best known estimate for $D_{\infty}$ in dimension $d=2$ with prime $q$ is $D_{\infty}(S_{a,q,2}, \mathrm{Vol}) \ll q^{-1/3}$, but it is not known whether the exponent $-1/3$ is optimal \cite[Remark 3.1]{HU2}.

\subsection{Modular parabolas}\label{modularparabola_section}

Given a positive integer $q$ and $a \in (\mathbb{Z}/q\mathbb{Z})^{\times}$, $b \in \mathbb{Z}/q \mathbb{Z}$, the modular parabola is defined as the set
\[ P_{a,b,q}= \left\{ \left( \frac{x}{q}, \frac{ax^2+bx}{q} \right)+\mathbb{Z}^2 : x \in \mathbb{Z}/q\mathbb{Z}  \right\} \subset \mathbb{R}^2 / \mathbb{Z}^2 \]
of cardinality $N=q$. Modular parabolas have applications in the Heilbronn triangle problem and the no-three-in-line problem in combinatorial geometry \cite{RO}.

The Fourier transform of the empirical measure of the set $P_{a,b,q}$ is a quadratic Gauss sum. Using a smoothing inequality stated in Section \ref{smoothing_section}, we will show that the same rate is attained in both the quadratic Wasserstein metric and the $L^2$ discrepancy with respect to balls as for i.i.d.\ random points. For the sake of simplicity, we only consider prime $q$.
\begin{thm}\label{modularparabola_theorem} Let $q$ be a prime, let $a \in (\mathbb{Z}/q\mathbb{Z})^{\times}$, $b \in \mathbb{Z}/q \mathbb{Z}$ and let $\mu_{a,b,q}=q^{-1} \sum_{p \in P_{a,b,q}} \delta_p$ denote the empirical measure of the set $P_{a,b,q}$. For any $0<r \le 1/2$, we have
\[ \begin{split} W_2 (\mu_{a,b,q}, \mathrm{Vol}) &\asymp (\log q)^{1/2} q^{-1/2}, \\ D_2(\mu_{a,b,q},\mathrm{Vol},r)^2 &= \frac{\pi r^2 - (16/3) r^3}{q} + O \left( \min \left\{ \frac{r^2}{q^2}, \frac{r}{q^3} \right\} \right), \\ D_2(\mu_{a,b,q},\mathrm{Vol})^2 &= \frac{\pi / 12 -1/6}{q} + O \left( \frac{1}{q^3} \right) \end{split} \]
with universal implied constants.
\end{thm}

It would be interesting to see whether the estimate for the Wasserstein metric can be refined to an asymptotic relation $W_2(\mu_{a,b,q}, \mathrm{Vol}) = (c+o(1)) (\log q)^{1/2} q^{-1/2}$ with some constant $c$, but this remains open. Note that the constants in the asymptotics for $D_2$ in Theorem \ref{modularparabola_theorem} differ from the random case \eqref{iidD2rate}.

The estimates in Theorem \ref{modularparabola_theorem} hold without averaging over parameters such as $a$, $b$ or $q$. The reason is that quadratic Gauss sums can be explicitly evaluated, and in particular have constant absolute value $\sqrt{q}$. In contrast, the Kloosterman sums that appear in the analysis of modular hyperbolas are simply bounded above by $\sqrt{q}$, and only attain the order of magnitude $\sqrt{q}$ on average.

\section{Approach}\label{approach_section}

This section discusses some general properties of the Wasserstein metric and the $L^2$ discrepancy with respect to balls that will play an important role in the proofs of the main results. Section \ref{WinftyWp_section} presents an upper estimate for $W_{\infty}$ in terms of $W_p$ with finite $p$ for empirical measures. Section \ref{smoothing_section} discusses upper and lower estimates for $W_p(\mu, \mathrm{Vol})$ in terms of the Fourier transform of $\mu$. Section \ref{D2harmonic_section} recalls an explicit formula for $D_2(\mu, \mathrm{Vol},r)$ in the torus in terms of the Fourier transform of $\mu$.

Let us start by recalling the behavior of the Wasserstein metric under Lipschitz maps.
\begin{lem}\label{lipschitz_lemma} Let $(X, \rho_X)$ and $(Y, \rho_Y)$ be compact metric spaces, and let $f: X \to Y$ be Lipschitz with Lipschitz constant $K \ge 0$, that is, $\rho_Y (f(x),f(y)) \le K \rho_X (x,y)$ for all $x,y \in X$. Let $\mu, \nu \in \mathcal{P}(X)$, and let $f_* \mu, f_* \nu \in \mathcal{P}(Y)$ denote the pushforward measures of $\mu$ and $\nu$. Then $W_p (f_* \mu, f_* \nu) \le K W_p (\mu, \nu)$ for all $1 \le p \le \infty$.
\end{lem}

\begin{proof} This is clear from the definition of the Wasserstein metric, see \cite[p.\ 110]{VI2}.
\end{proof}

\subsection{Estimating $W_{\infty}$ in terms of $W_p$}\label{WinftyWp_section}

We present a converse to the general inequality $W_p (\mu, \nu) \le W_q (\mu, \nu)$, $1 \le p \le q \le \infty$, in the special case when $\mu$ is supported on $N$ points with equal weights and $\nu$ is the uniform measure. In particular, we show that $W_p(\mu_N, \mathrm{Vol})$ with $p=\log N$ is already as large as $W_{\infty}(\mu_N, \mathrm{Vol})$ up to a constant factor. An upper estimate for $W_p(\mu_N, \mathrm{Vol})$ with explicit dependence on $1 \le p < \infty$ thus implies an upper estimate for $W_{\infty}(\mu_N, \mathrm{Vol})$ simply by substituting $p=\log N$.

\begin{lem}\label{WqWplemma} Let $M=\mathbb{R}^d/\mathbb{Z}^d$ or $\mathbb{S}^d$, let $a_1, a_2, \ldots, a_N \in M$ be arbitrary points and let $\mu_N=N^{-1} \sum_{n=1}^N \delta_{a_n}$. For any $1 \le p \le q \le \infty$, we have
\[ N^{1/q} W_q (\mu_N, \mathrm{Vol}) \le 2 N^{1/p} W_p (\mu_N, \mathrm{Vol}) . \]
In particular,
\[ W_{\infty} (\mu_N, \mathrm{Vol}) \le 2e W_{\log N} (\mu_N, \mathrm{Vol}) . \]
\end{lem}

\begin{proof} Let $1 \le p \le q < \infty$. The Wasserstein distance from $\mu_N$ to $\nu_N=N^{-1} \sum_{n=1}^N \delta_{b_n}$ with arbitrary points $b_1, b_2, \ldots, b_N \in M$ can be expressed as
\[ W_p (\mu_N, \nu_N) = \min_{\sigma \in S_N} \left( \frac{1}{N} \sum_{n=1}^N \rho_M (a_n, b_{\sigma(n)})^p \right)^{1/p} , \]
where the minimum is over all bijections $\sigma: \{ 1, 2, \ldots, N \} \to \{ 1, 2, \ldots, N \}$, see \cite[p.\ 5]{VI1}. Let $\sigma$ be the bijection for which the minimum is attained in the previous formula. An application of the general inequality $(\sum_{n=1}^N x_n)^{p/q} \le \sum_{n=1}^N x_n^{p/q}$, $x_n \ge 0$, with the exponent $0<p/q \le 1$ leads to
\begin{equation}\label{WqWpinequality}
N^{1/q} W_q(\mu_N, \nu_N) \le \left( \sum_{n=1}^N \rho_M (a_n, b_{\sigma(n)})^q \right)^{1/q} \le \left( \sum_{n=1}^N \rho_M (a_n, b_{\sigma(n)})^p \right)^{1/p} = N^{1/p} W_p (\mu_N, \nu_N) .
\end{equation}

If $M=\mathbb{R}^d / \mathbb{Z}^d$, then let $U$ be a uniformly distributed random variable in $\mathbb{R}^d/\mathbb{Z}^d$, and let $b_n=a_n+U$. If $M=\mathbb{S}^d$, then let $U$ be a uniformly distributed random variable in $\mathrm{SO}(d+1)$, and let $b_n=Ua_n$. In both cases $b_n$ is a uniformly distributed random variable in $M$ for each $1 \le n \le N$, hence $\mathbb{E} \nu_N = \mathrm{Vol}$. As the map $x \mapsto x+U$ on $\mathbb{R}^d/\mathbb{Z}^d$ resp.\ $x \mapsto Ux$ on $\mathbb{S}^d$ is an isometry that preserves the uniform measure $\mathrm{Vol}$, Lemma \ref{lipschitz_lemma} yields $W_p (\nu_N, \mathrm{Vol})=W_p(\mu_N, \mathrm{Vol})$. Using the convexity of the optimal cost \cite[Theorem 4.8]{VI2}, inequality \eqref{WqWpinequality} and the triangle inequality for the Wasserstein metric we thus deduce
\[ \begin{split} N^{1/q} W_q (\mu_N, \mathrm{Vol}) &= N^{1/q} W_q (\mu_N, \mathbb{E}\nu_N) \le \mathbb{E} (N^{1/q} W_q(\mu_N, \nu_N)) \le \mathbb{E} (N^{1/p} W_p (\mu_N, \nu_N)) \\ &\le \mathbb{E} (N^{1/p} W_p (\mu_N, \mathrm{Vol}) + N^{1/p} W_p(\nu_N, \mathrm{Vol})) = 2 N^{1/p} W_p (\mu_N, \mathrm{Vol}) . \end{split} \]
This finishes the proof in the case $1 \le p \le q < \infty$. Taking the limit as $q \to \infty$ establishes the claim for $q=\infty$. In particular, choosing $p=\log N$ and $q=\infty$ leads to $W_{\infty}(\mu_N, \mathrm{Vol}) \le 2 N^{1/\log N}W_{\log N}(\mu_N, \mathrm{Vol}) = 2 e W_{\log N}(\mu_N, \mathrm{Vol})$, as claimed.
\end{proof}

\begin{remark} Lemma \ref{WqWplemma} holds in any compact homogeneous space $M=G/K$ equipped with a $G$-invariant metric $\rho_{G/K}$ and a $G$-invariant probability measure $\mathrm{Vol}$. The proof then uses a random variable $U$ whose distribution is the normalized Haar measure on $G$.
\end{remark}

\subsection{Wasserstein metric and harmonic analysis}\label{smoothing_section}

The main tool in the present paper is an estimate for $W_p (\mu, \mathrm{Vol})$ in terms of the Fourier transform of $\mu$. In the torus, the Fourier transform of $\mu \in \mathcal{P}(\mathbb{R}^d/\mathbb{Z}^d)$ is defined as $\widehat{\mu}(k)=\int_{\mathbb{R}^d/\mathbb{Z}^d} e^{-2 \pi i \langle k, x \rangle} \, \mathrm{d}\mu(x)$, $k \in \mathbb{Z}^d$. The following two theorems are special cases of two smoothing inequalities due to the first author and Cuenin.
\begin{thm}[{\cite[Theorem 1]{BC}}]\label{torussmoothing_theorem1} Let $\mu \in \mathcal{P}(\mathbb{R}^d/\mathbb{Z}^d)$ and let $2 \le p < \infty$ and $1<q\le 2$ be such that $1/p+1/q=1$. For any integer $H>(p+2)/2$, we have
\[ W_p(\mu, \mathrm{Vol}) \le \frac{(p+4)d^{1/2}}{6^{1/2}(H-(p+2)/2)} + p \bigg( \sum_{\substack{k \in [-H,H]^d \\ k \neq 0}} \frac{|\widehat{\mu}(k)|^q}{(2 \pi |k|)^q} \bigg)^{1/q} . \]
\end{thm}

\begin{thm}[{\cite[Theorem 2]{BC}}]\label{torussmoothing_theorem2} Let $\mu \in \mathcal{P}(\mathbb{R}^d/\mathbb{Z}^d)$ and assume there exist constants $b \ge 0$ and $r>0$ such that $\mu (B(x,r)) \ge b$ for all $x \in \mathbb{R}^d/\mathbb{Z}^d$. Let $2 \le p < \infty$ and $1< q \le 2$ be such that $1/p+1/q=1$. For any real $t>0$, we have
\[ W_p (\mu, \mathrm{Vol}) \le (2d+p)^{1/2} t^{1/2} + \frac{p(1-\delta^{1/p})}{1-\delta} \bigg( \sum_{\substack{k \in \mathbb{Z}^d \\ k \neq 0}} e^{-4 \pi^2 |k|^2 q t} \frac{|\widehat{\mu}(k)|^q}{(2 \pi |k|)^q} \bigg)^{1/q}, \]
where
\[ \delta = \min \left\{ b \frac{e^{-r^2/(4t)}}{(4 \pi t)^{d/2}}, 1 \right\} . \]
\end{thm}

The eigenvalues of the Laplace--Beltrami operator $\Delta$ on the sphere $\mathbb{S}^d$ are $-\lambda_{\ell}=-\ell (\ell+d-1)$, $\ell \ge 0$, and $-\lambda_{\ell}$ has multiplicity $d_{\ell}=\binom{\ell+d}{d}-\binom{\ell+d-2}{d}$. Let $Y_{\ell}^m$, $1 \le m \le d_{\ell}$, be an arbitrary orthonormal basis in the eigenspace $\{ \phi \in L^2 (\mathbb{S}^2, \mathrm{Vol}): \Delta \phi = -\lambda_{\ell} \phi  \}$, for example the standard spherical harmonics. The Fourier transform of $\mu \in \mathcal{P}(\mathbb{S}^d)$ is then defined as $\widehat{\mu}(\ell, m) = \int_{\mathbb{S}^d} \overline{Y_{\ell}^m} \, \mathrm{d}\mu$, $\ell \ge 0$, $1 \le m \le d_{\ell}$, where the overline denotes the complex conjugate. The following special case of a smoothing inequality due to the first author and Cuenin is the analogue of Theorem \ref{torussmoothing_theorem1} on the sphere.

\begin{thm}[{\cite[Theorem 7]{BC}}]\label{spheresmoothing_theorem} Let $\mu \in \mathcal{P}(\mathbb{S}^d)$ and let $2 \le p < \infty$ and $1<q \le 2$ be such that $1/p+1/q=1$. For any integer $L \ge 1$, we have
\[ W_p (\mu, \mathrm{Vol}) \le 8d12^{1/p}e^{d/p}\left( \frac{p+5}{2} \right)^{1+d/p} \frac{1}{L} + 2(p-1)p \bigg( \sum_{\ell=1}^L \frac{d_{\ell}^{1-q/2}}{\lambda_{\ell}^{q/2}} \bigg( \sum_{m=1}^{d_{\ell}} |\widehat{\mu}(\ell,m)|^2 \bigg)^{q/2} \bigg)^{1/q} . \]
\end{thm}

Finding a lower bound for the quadratic Wasserstein metric $W_2$ relies on the following lemma due to Ledoux. We mention that it actually applies to all compact, connected Riemannian manifolds without boundary with a positive semidefinite Ricci curvature.
\begin{lem}\label{ledoux_lemma}  Let $M=\mathbb{R}^d / \mathbb{Z}^d$ or $\mathbb{S}^d$, and let $\mu \in \mathcal{P}(M)$ be absolutely continuous with density $\frac{\mathrm{d}\mu}{\mathrm{d}\mathrm{Vol}}=1+f$, where $f$ is a smooth function such that $\sup_{x \in M}|f(x)|\le c$ with some constant $0<c \le 1$. Then
\[ W_2 (\mu, \mathrm{Vol}) = \left( 1+O(c) \right) \left( \int_M f (\Delta^{-1} f) \, \mathrm{d}\mathrm{Vol} \right)^{1/2} \]
with a universal implied constant.
\end{lem}

\begin{proof} In an unpublished sequel to \cite{LE}, Ledoux showed that
\[ \sqrt{2-\frac{e^c-1}{c}} \left( \int_M f (\Delta^{-1} f) \, \mathrm{d}\mathrm{Vol} \right)^{1/2} \le W_2 (\mu, \mathrm{Vol}) \le \frac{2(1-\sqrt{1-c})}{c} \left( \int_M f (\Delta^{-1} f) \, \mathrm{d}\mathrm{Vol} \right)^{1/2}, \]
where by Green's formula,
\[ \begin{split} \left( \int_M f (\Delta^{-1} f) \, \mathrm{d}\mathrm{Vol} \right)^{1/2} &= \left( \int_M |\nabla (\Delta^{-1} f)|^2 \, \mathrm{d}\mathrm{Vol} \right)^{1/2} \\ &= \sup \left\{ \left| \int_M \langle \nabla (\Delta^{-1} f), \nabla g \rangle \, \mathrm{d}\mathrm{Vol} \right| : \int_M |\nabla g|^2 \, \mathrm{d}\mathrm{Vol} \le 1 \right\} \\ &= \sup \left\{ \left| \int_M f g \, \mathrm{d} \mathrm{Vol} \right| : \int_M |\nabla g|^2 \, \mathrm{d}\mathrm{Vol} \le 1 \right\} \\ &= \sup \left\{ \left| \int_M g \, \mathrm{d}(\mu-\mathrm{Vol}) \right| : \int_M |\nabla g|^2 \, \mathrm{d}\mathrm{Vol} \le 1 \right\} \\ &= \| \mu - \mathrm{Vol} \|_{\dot{H}_{-1}} \end{split} \]
is in fact a certain dual Sobolev norm of the signed measure $\mu - \mathrm{Vol}$. The same upper bound in a more general form appeared also in \cite[Lemma 13]{BC}, whereas the lower bound with a different constant
\[ \frac{\log (1+c)}{2 (\sqrt{1+c}-1)} \int_M f (\Delta^{-1} f) \, \mathrm{d}\mathrm{Vol} \le W_2 (\mu, \mathrm{Vol}) \]
was proved earlier in \cite[Theorem 2.5]{PE}. All three constants
\[ \sqrt{2-\frac{e^c-1}{c}}, \quad \frac{\log (1+c)}{2 (\sqrt{1+c}-1)}, \quad \frac{2(1-\sqrt{1-c})}{c} \]
are $1+O(c)$ for $0<c \le 1$, which proves the claim.
\end{proof}

\subsection{$L^2$ discrepancy and harmonic analysis}\label{D2harmonic_section}

It is well known that the $L^2$ discrepancy with respect to balls in the torus can be expressed in terms of the Fourier transform. We recall the proof for the convenience of the reader. Throughout, $J_{\alpha}$ denotes the Bessel function of the first kind and $\Gamma$ is the gamma function.

\begin{lem}\label{D2torus_lemma} Let $\mu \in \mathcal{P}(\mathbb{R}^d / \mathbb{Z}^d)$. For any $0<r \le 1/2$, we have
\[ D_2(\mu, \mathrm{Vol}, r) = \bigg( \sum_{\substack{k \in \mathbb{Z}^d \\ k \neq 0}} b_k(r) |\widehat{\mu}(k)|^2 \bigg)^{1/2} \quad \textrm{and} \quad D_2(\mu, \mathrm{Vol}) = \bigg( \sum_{\substack{k \in \mathbb{Z}^d \\ k \neq 0}} b_k |\widehat{\mu}(k)|^2 \bigg)^{1/2} \]
with coefficients
\[ \begin{split} b_k(r) &=\frac{r^d}{|k|^d} J_{d/2}(2 \pi r |k|)^2 \ll \min \left\{ r^{2d}, \frac{r^{d-1}}{|k|^{d+1}} \right\}, \\ b_k &= 2 \int_0^{1/2} b_k(r) \, \mathrm{d}r = \frac{1}{d 2^d \pi^2 |k|^{d+1}} + O \left( \frac{1}{|k|^{d+2}} \right) . \end{split} \]
The total sums of the coefficients are
\[ \begin{split} \sum_{\substack{k \in \mathbb{Z}^d \\ k \neq 0}} b_k(r) &= \frac{\pi^{d/2}r^d}{\Gamma \left( \frac{d}{2}+1 \right)} \left( 1- \frac{\pi^{d/2}r^d}{\Gamma \left( \frac{d}{2}+1 \right)} \right), \\ \sum_{\substack{k \in \mathbb{Z}^d \\ k \neq 0}} b_k &= \frac{\pi^{d/2}}{(d+1)2^d \Gamma \left( \frac{d}{2}+1 \right)} - \frac{\pi^d}{(2d+1)2^{2d} \Gamma \left( \frac{d}{2} +1 \right)^2} . \end{split} \]
\end{lem}

\begin{proof} Let $V(r)=\mathrm{Vol}(B(0,r))=(\pi^{d/2}/\Gamma (d/2+1))r^d$. Observe that for all $y \in \mathbb{R}^d / \mathbb{Z}^d$, $k \in \mathbb{Z}^d$ and $0<r \le 1/2$, we have
\begin{equation}\label{ballfouriertransform}
\begin{split} \int_{\mathbb{R}^d / \mathbb{Z}^d} &\left( \mathds{1}_{B(x,r)}(y) - V(r) \right) e^{- 2 \pi i \langle k, x \rangle} \, \mathrm{d}\mathrm{Vol}(x) \\ &= \int_{\mathbb{R}^d / \mathbb{Z}^d} \left( \mathds{1}_{B(0,r)}(x-y) - V(r) \right) e^{- 2 \pi i \langle k, x \rangle} \, \mathrm{d}\mathrm{Vol}(x) \\ &= \left\{ \begin{array}{ll} 0 & \textrm{if } k=0, \\ e^{- 2 \pi i \langle k, y \rangle} r^{d/2} |k|^{-d/2} J_{d/2}(2 \pi r |k|) & \textrm{if } k \neq 0. \end{array} \right. \end{split}
\end{equation}
Therefore the Fourier coefficients of the function $\mu (B(x,r)) - \mathrm{Vol}(B(x,r))$ are
\[ \begin{split} \int_{\mathbb{R}^d / \mathbb{Z}^d} &\left( \mu (B(x,r)) - \mathrm{Vol}(B(x,r)) \right) e^{-2 \pi i \langle k, x \rangle} \, \mathrm{d}\mathrm{Vol}(x) \\ &= \int_{\mathbb{R}^d / \mathbb{Z}^d} \int_{\mathbb{R}^d / \mathbb{Z}^d} \left( \mathds{1}_{B(x,r)}(y) - V(r) \right) e^{- 2 \pi i \langle k, x \rangle} \, \mathrm{d}\mathrm{Vol}(x) \, \mathrm{d}\mu (y) \\ &= \left\{ \begin{array}{ll} 0 & \textrm{if } k=0, \\ \widehat{\mu}(k) r^{d/2} |k|^{-d/2} J_{d/2}(2 \pi r |k|) & \textrm{if } k \neq 0. \end{array} \right. \end{split} \]
An application of the Parseval formula thus yields
\[ D_2(\mu, \mathrm{Vol},r)^2 = \int_{\mathbb{R}^d / \mathbb{Z}^d} \left( \mu (B(x,r)) - \mathrm{Vol}(B(x,r)) \right)^2 \, \mathrm{d}\mathrm{Vol}(x) = \sum_{\substack{k \in \mathbb{Z}^d \\ k \neq 0}} b_k(r) |\widehat{\mu}(k)|^2 \]
with coefficients $b_k(r) = r^d |k|^{-d} J_{d/2}(2 \pi r |k|)^2$, as claimed.

The asymptotic bounds $J_{d/2}(t) \ll t^{d/2}$ as $t \to 0^+$ and $J_{d/2}(t) \ll t^{-1/2}$ as $t \to \infty$ \cite[p.\ 209 and 218]{AAR} imply $J_{d/2}(2 \pi r |k|)^2 \ll \min \{ r^d |k|^d , r^{-1} |k|^{-1} \}$ and the desired upper estimate for $b_k(r)$ follows. Setting $y=0$ in \eqref{ballfouriertransform} and applying the Parseval formula leads to the total sum of the coefficients
\[ \sum_{\substack{k \in \mathbb{Z}^d \\ k \neq 0}} b_k(r) = \sum_{\substack{k \in \mathbb{Z}^d \\ k \neq 0}} \frac{r^{d}}{|k|^d} J_{d/2}(2 \pi r |k|)^2 = \int_{\mathbb{R}^d / \mathbb{Z}^d} \left( \mathds{1}_{B(0,r)}(x) - V(r) \right)^2 \, \mathrm{d}\mathrm{Vol}(x) = V(r)(1-V(r)). \]

A further averaging over $0<r \le 1/2$ yields $D_2(\mu, \mathrm{Vol})=(\sum_{k \neq 0} b_k |\widehat{\mu}(k)|^2)^{1/2}$ with coefficients
\[ b_k = 2 \int_0^{1/2} b_k(r) \, \mathrm{d}r = 2 \int_0^{1/2} \frac{r^d}{|k|^d} J_{d/2}(2 \pi r |k|)^2 \, \mathrm{d}r , \]
as claimed. In particular, the total sum of the coefficients $b_k$ is
\[ \begin{split} \sum_{\substack{k \in \mathbb{Z}^d \\ k \neq 0}} b_k &= 2 \int_0^{1/2} \frac{\pi^{d/2}r^d}{\Gamma \left( \frac{d}{2}+1 \right)} \left( 1- \frac{\pi^{d/2}r^d}{\Gamma \left( \frac{d}{2}+1 \right)} \right) \, \mathrm{d}r \\ &=\frac{\pi^{d/2}}{(d+1)2^d \Gamma \left( \frac{d}{2}+1 \right)} - \frac{\pi^d}{(2d+1)2^{2d} \Gamma \left( \frac{d}{2} +1 \right)^2}. \end{split} \]
Finally, the asymptotics of the Bessel function \cite[p.\ 209]{AAR}
\[ J_{d/2}(t)= \sqrt{\frac{2}{\pi t}} \left( \cos \left( t-\frac{(d+1)\pi}{4} \right) + O \left( \frac{1}{t} \right) \right) \quad \textrm{as } t \to \infty \]
shows that
\[ \begin{split} b_k &= 2 \int_0^{1/2} \frac{r^d}{|k|^d} \cdot \frac{1}{\pi^2 r |k|} \cos^2 \left( 2 \pi r |k| - \frac{(d+1)\pi}{4} \right) \, \mathrm{d}r + O \left( \int_0^{1/2} \frac{r^d}{|k|^d} \cdot \frac{1}{r^2 |k|^2} \, \mathrm{d} r \right) \\ &= \frac{1}{d 2^d \pi^2 |k|^{d+1}} + O \left( \frac{1}{|k|^{d+2}} \right) . \end{split} \]
\end{proof}

The empirical measure $\mu_N$ of a sample of $N$ i.i.d.\ uniformly distributed random points $a_1, a_2, \ldots, a_N \in \mathbb{R}^d / \mathbb{Z}^d$ satisfies
\[ \mathbb{E} (|\widehat{\mu_N}|^2) = \frac{1}{N^2} \mathbb{E} \left( \left| \sum_{n=1}^N e^{- 2 \pi i \langle k, a_n \rangle} \right|^2 \right) =\frac{1}{N} + \frac{1}{N^2} \sum_{\substack{n,m=1 \\ n \neq m}}^N \mathbb{E} \left( e^{2 \pi i \langle k, a_n \rangle} e^{-2 \pi i \langle k, a_m \rangle} \right) = \frac{1}{N} . \]
Substituting the previous formula in Lemma \ref{D2torus_lemma} leads to
\[ \begin{split} \mathbb{E} \left( D_2 (\mu_N, \mathrm{Vol},r)^2 \right) &= \frac{\pi^{d/2}r^d}{\Gamma \left( \frac{d}{2}+1 \right)} \left( 1- \frac{\pi^{d/2}r^d}{\Gamma \left( \frac{d}{2}+1 \right)} \right) \frac{1}{N} , \\ \mathbb{E} \left( D_2 (\mu_N, \mathrm{Vol})^2 \right) &= \left( \frac{\pi^{d/2}}{(d+1)2^d \Gamma \left( \frac{d}{2}+1 \right)} - \frac{\pi^d}{(2d+1)2^{2d} \Gamma \left( \frac{d}{2} +1 \right)^2} \right) \frac{1}{N} . \end{split} \]
This proves relation \eqref{iidD2rate}.

\section{Wasserstein metric and $L^2$ discrepancy}\label{wassersteinL2discrepancy_section}

\begin{proof}[Proof of Theorem \ref{W2D2_theorem}] Let us first prove the theorem for $M=\mathbb{R}^d / \mathbb{Z}^d$. Lemma \ref{D2torus_lemma} expresses $D_2(\mu, \mathrm{Vol})$ in terms of the Fourier transform as
\[ D_2 (\mu, \mathrm{Vol}) \asymp \bigg( \sum_{\substack{k \in \mathbb{Z}^d \\ k \neq 0}} \frac{|\widehat{\mu}(k)|^2}{|k|^{d+1}} \bigg)^{1/2} . \]
An application of Theorem \ref{torussmoothing_theorem1} with $p=2$ shows that for any integer $H>2$,
\[ \begin{split} W_2 (\mu, \mathrm{Vol}) &\ll \frac{1}{H} + \bigg( \sum_{\substack{k \in [-H,H]^d \\ k \neq 0}} \frac{|\widehat{\mu}(k)|^2}{|k|^2} \bigg)^{1/2} \ll \frac{1}{H} + \bigg( \sum_{\substack{k \in [-H,H]^d \\ k \neq 0}} \frac{|\widehat{\mu}(k)|^2}{|k|^{d+1}} H^{d-1} \bigg)^{1/2} \\ &\ll \frac{1}{H} + D_2 (\mu, \mathrm{Vol}) H^{(d-1)/2} . \end{split} \]
The optimal choice $H \asymp D_2(\mu, \mathrm{Vol})^{-2/(d+1)}$ leads to $W_2(\mu, \mathrm{Vol}) \ll D_2(\mu, \mathrm{Vol})^{2/(d+1)}$, as claimed.

We now prove the theorem for $M=\mathbb{S}^d$. Since $V(r)=\mathrm{Vol}(B(x,r))$ does not depend on $x$, we have
\[ \begin{split} D_2(\mu, \mathrm{Vol})^2 &= \frac{1}{2} \int_0^{\pi} \int_{\mathbb{S}^d} |\mu (B(x,r)) - V(r)|^2 \sin r \, \mathrm{d} \mathrm{Vol}(x) \, \mathrm{d} r \\ &= \frac{1}{2} \int_0^{\pi} \int_{\mathbb{S}^d} \left| \int_{\mathbb{S}^d} \left( \mathds{1}_{B(x,r)}(y) - V(r) \right) \mathrm{d}\mathrm{\mu}(y) \right|^2 \sin r \, \mathrm{d}\mathrm{Vol}(x) \, \mathrm{d}r \\ &= \frac{1}{2} \int_{\mathbb{S}^d} \int_{\mathbb{S}^d} \int_0^{\pi} \int_{\mathbb{S}^d} \left( \mathds{1}_{B(y_1, r)}(x) - V(r) \right) \left( \mathds{1}_{B(y_2, r)}(x) -V(r) \right)  \sin r \, \mathrm{d} \mathrm{Vol}(x) \, \mathrm{d} r \, \mathrm{d}\mu(y_1) \mathrm{d}\mu (y_2) \\ &= \frac{1}{2} \int_{\mathbb{S}^d} \int_{\mathbb{S}^d} \int_0^{\pi} \left( \mathrm{Vol}(B(y_1, r) \cap B(y_2, r)) - V(r)^2 \right) \sin r \, \mathrm{d} r \, \mathrm{d}\mu(y_1) \, \mathrm{d}\mu(y_2) . \end{split} \]
Let $P_{\ell}^{(\alpha, \beta)}(t)$ denote the Jacobi polynomials. Skriganov \cite[Theorem 4.1 (ii)]{SK} showed the expansion
\[ \int_0^{\pi} \left( \mathrm{Vol}(B(y_1, r) \cap B(y_2, r)) - V(r)^2 \right) \sin r \, \mathrm{d}r = \sum_{\ell=1}^{\infty} b_{\ell} P_{\ell}^{(d/2-1, d/2-1)}(\langle y_1, y_2 \rangle) \]
with positive real coefficients
\[ b_{\ell} = \frac{\Gamma (d) (2\ell+d-1) \Gamma (\ell)^2 \Gamma (\ell+d-1)}{2^{2d} \Gamma \left( \frac{d}{2} \right) \Gamma \left( \ell + \frac{d}{2} \right)^3} \int_0^{\pi} (\sin r)^{2d+1} P_{\ell-1}^{(d/2, d/2)} (\cos r)^2 \, \mathrm{d}r \asymp \frac{1}{\ell^{d/2+1}} . \]
The order of magnitude of $b_{\ell}$ follows from the asymptotic formula of Darboux for Jacobi polynomials, see Szeg\H{o} \cite[p.\ 173 and p.\ 196]{SZ}. In particular,
\begin{equation}\label{D2expansion}
D_2 (\mu, \mathrm{Vol})^2 \asymp \sum_{\ell=1}^{\infty} \frac{1}{\ell^{d/2+1}} \int_{\mathbb{S}^d} \int_{\mathbb{S}^d} P_{\ell}^{(d/2-1, d/2-1)} (\langle y_1, y_2 \rangle) \, \mathrm{d}\mu(y_1) \, \mathrm{d}\mu (y_2) .
\end{equation}

The addition formula for spherical harmonics
\[ \sum_{m=1}^{d_{\ell}} Y_{\ell}^m (y_1) \overline{Y_{\ell}^m(y_2)} = \frac{(2\ell+d-1)\Gamma \left( \frac{d}{2} \right) \Gamma (\ell+d-1)}{\Gamma(d) \Gamma \left( \ell + \frac{d}{2} \right)} P_{\ell}^{(d/2-1, d/2-1)} (\langle y_1, y_2 \rangle), \qquad y_1, y_2 \in \mathbb{S}^d, \]
shows that
\[ \begin{split} \sum_{m=1}^{d_{\ell}} |\widehat{\mu}(\ell, m)|^2 &= \int_{\mathbb{S}^d} \int_{\mathbb{S}^d} \sum_{m=1}^{d_{\ell}} Y_{\ell}^m (y_1) \overline{Y_{\ell}^m (y_2)} \, \mathrm{d}\mu(y_1) \, \mathrm{d} \mu(y_2) \\ &= \frac{(2\ell+d-1)\Gamma \left( \frac{d}{2} \right) \Gamma (\ell+d-1)}{\Gamma(d) \Gamma \left( \ell + \frac{d}{2} \right)} \int_{\mathbb{S}^d} \int_{\mathbb{S}^d} P_{\ell}^{(d/2-1, d/2-1)}(\langle y_1, y_2 \rangle) \, \mathrm{d}\mu(y_1) \, \mathrm{d}\mu(y_2) \\ &\asymp \ell^{d/2} \int_{\mathbb{S}^d} \int_{\mathbb{S}^d} P_{\ell}^{(d/2-1, d/2-1)}(\langle y_1, y_2 \rangle) \, \mathrm{d}\mu(y_1) \, \mathrm{d}\mu(y_2) . \end{split} \]
An application of Theorem \ref{spheresmoothing_theorem} with $p=2$, \eqref{D2expansion} and the previous formula thus show that for any positive integer $L$,
\[ \begin{split} W_2 (\mu, \mathrm{Vol}) &\ll \frac{1}{L} + \bigg( \sum_{\ell=1}^L \frac{1}{\ell^2} \sum_{m=1}^{d_{\ell}} |\widehat{\mu}(\ell,m)|^2 \bigg)^{1/2} \\ &\ll \frac{1}{L} + \left( \sum_{\ell=1}^L \frac{L^{d-1}}{\ell^{d/2+1}} \int_{\mathbb{S}^d} \int_{\mathbb{S}^d} P_{\ell}^{(d/2-1, d/2-1)}(\langle y_1, y_2 \rangle) \, \mathrm{d}\mu(y_1) \, \mathrm{d}\mu(y_2) \right)^{1/2} \\ &\ll \frac{1}{L} + D_2(\mu, \mathrm{Vol}) L^{(d-1)/2} . \end{split} \]
The optimal choice $L \asymp D_2(\mu, \mathrm{Vol})^{-2/(d+1)}$ leads to $W_2 (\mu, \mathrm{Vol}) \ll D_2 (\mu, \mathrm{Vol})^{2/(d+1)}$, as claimed.
\end{proof}

\section{Equidistribution in the torus}\label{torus_section}

\subsection{Kronecker sequences and higher rank lattices}

We will need a quantitative form of Kronecker's theorem due to Cassels. The following lemma in the special case $N_1=N_2=\cdots=N_m$ is stated in \cite[p.\ 99]{CA}, but the proof actually works for general $N_1, N_2, \ldots, N_m \ge 2$. For the sake of completeness, we discuss the necessary modification of the proof and deduce the general case of the lemma.
\begin{lem}\label{cassels_lemma} Let $\alpha_1, \alpha_2, \ldots, \alpha_m \in \mathbb{R}^d$, let $N_1, N_2, \ldots, N_m \ge 2$ be integers and let $N=N_1 N_2 \cdots N_m$. Under the assumption \eqref{diophantinelattice}, every axis-parallel cube in $\mathbb{R}^d / \mathbb{Z}^d$ of sidelength
\[ \frac{((d+m)!)^2}{2^{d+m-1} \psi^{-1} \left( 2^{(d+m-2)m} ((d+m)!)^{-2m} I N \right)} \]
contains at least one point from the set
\begin{equation}\label{latticeset}
\left\{ n_1 \alpha_1 + n_2 \alpha_2 + \cdots + n_m \alpha_m + \mathbb{Z}^d : n_1 \in [1, N_1], \,\, n_2 \in [1, N_2], \,\, \ldots, \,\, n_m \in [1,N_m] \right\} .
\end{equation}
\end{lem}

\begin{proof} Let $\beta \in \mathbb{R}^d$ be arbitrary and let
\[ C = \frac{((d+m)!)^2}{2^{d+m} \psi^{-1} \left( 2^{(d+m-2)m} ((d+m)!)^{-2m} I N \right)} \]
be half the desired sidelength. Let $\alpha_j=(\alpha_{j1}, \alpha_{j2}, \ldots, \alpha_{jd})$, $1 \le j \le m$. Consider the linear forms
\[ \begin{split} f_i (n_1, \ldots, n_m, y_1, \ldots, y_d) &= \left\{ \begin{array}{ll} C^{-1} \left( \sum_{j=1}^m n_j \alpha_{ji} + y_i \right) & \textrm{if } 1 \le i \le d, \\ \left( \frac{N_{i-d}-1}{2} \right)^{-1} n_{i-d} & \textrm{if } d< i \le d+m, \end{array} \right. \\ g_i (v_1, \ldots, v_m, k_1, \ldots, k_d) &= \left\{ \begin{array}{ll} C k_i & \textrm{if } 1 \le i \le d, \\ \left( \frac{N_{i-d}-1}{2} \right) \left( v_{i-d} - \langle k, \alpha_{i-d} \rangle \right) & \textrm{if } d< i \le d+m, \end{array} \right. \end{split} \]
and note that $\sum_{i=1}^{d+m} f_i g_i = \sum_{i=1}^m n_i v_i + \sum_{i=1}^d y_i k_i$. We claim that for all integers $v_1, \ldots, v_m$, $k_1, \ldots, k_d$, we have
\begin{equation}\label{cassels}
\left\| \sum_{i=1}^d k_i \beta_i \right\|_{\mathbb{R}/\mathbb{Z}} \le \frac{2^{d+m-1}}{((d+m)!)^2} \max \left\{ C \| k \|_{\infty}, \max_{1 \le j \le m} \frac{N_j-1}{2} \left| v_j - \langle k, \alpha_j \rangle \right| \right\} .
\end{equation}
Indeed, if $k=0$, then the left-hand side vanishes, consequently \eqref{cassels} holds. Otherwise the left-hand side is at most $1/2$, and by assumption \eqref{diophantinelattice},
\[ \max_{1 \le j \le m} \frac{N_j-1}{2} \left| v_j - \langle k, \alpha_j \rangle \right| \ge \left( \prod_{1 \le j \le m} \frac{N_j}{4} \| \langle k, \alpha_j \rangle \|_{\mathbb{R}/\mathbb{Z}} \right)^{1/m} \ge \frac{I^{1/m} N^{1/m}}{4 \psi (\| k \|_{\infty})^{1/m}} . \]
The choice of $C$ ensures that for any $k \in \mathbb{Z}^d$, $k \neq 0$, either
\[ \frac{1}{2} \le \frac{2^{d+m-1}}{((d+m)!)^2} C \| k \|_{\infty}  \]
or
\[ \frac{1}{2} \le \frac{2^{d+m-1}}{((d+m)!)^2} \cdot \frac{I^{1/m} N^{1/m}}{4 \psi (\| k \|_{\infty})^{1/m}} \]
holds, and \eqref{cassels} follows. An application of \cite[Theorem XVI, p.\ 97]{CA} then shows that there exist integers $n_1, \ldots, n_m, y_1, \ldots, y_d$ such that
\[ \begin{aligned} \bigg| \sum_{j=1}^m n_j \alpha_{ji} + y_i - \beta_i \bigg| &\le C, & &1 \le i \le d, \\ |n_j| &\le \frac{N_j-1}{2}, & &1 \le j \le m . \end{aligned} \]
In particular, the point $\sum_{j=1}^m n_j \alpha_j + \mathbb{Z}^d$ lies in the axis-parallel cube centered at $\beta+\mathbb{Z}^d$ with sidelength $2C$. Since a translated copy of the set
\[ \left\{ n_1 \alpha_1 + n_2 \alpha_2 + \cdots + n_m \alpha_m + \mathbb{Z}^d : |n_1| \le \frac{N_1-1}{2}, \,\, |n_2| \le \frac{N_2-1}{2}, \,\, \ldots, |n_m| \le \frac{N_m-1}{2} \right\} \]
is contained in the set \eqref{latticeset}, the latter also intersects all axis-parallel cubes of sidelength $2C$, as claimed.
\end{proof}

\begin{proof}[Proof of Theorem \ref{lattice_theorem}] Throughout the proof, implied constants depend only on $I$, $c$, $d$ and $m$. The claim for $W_{\infty}$ follows from substituting $p=\log N$ in the claim for $W_p$ and using Lemma \ref{WqWplemma}.

Let $2 \le p < \infty$ and $1 < q \le 2$ be such that $1/p+1/q=1$, and let $b \ge 0$, $r>0$ and $t>0$ to be chosen. The Fourier transform of
\[ \mu_N = \frac{1}{N} \sum_{n_1=1}^{N_1} \sum_{n_2=1}^{N_2} \cdots \sum_{n_m=1}^{N_m} \delta_{n_1 \alpha_1 + n_2 \alpha_2 + \cdots + n_m \alpha_m + \mathbb{Z}^d} \]
can be estimated as
\[ |\widehat{\mu_N}(k)| = \left| \frac{1}{N} \prod_{j=1}^m \sum_{n_j=1}^{N_j} e^{- 2 \pi i n_j \langle k, \alpha_j \rangle}  \right| \le \frac{1}{N} \prod_{j=1}^m \frac{2}{|1-e^{-2 \pi i \langle k, \alpha_j \rangle}|} \ll \frac{1}{N \prod_{j=1}^m \| \langle k, \alpha_j \rangle \|_{\mathbb{R}/\mathbb{Z}}}. \]
An application of Theorem \ref{torussmoothing_theorem2} thus yields
\begin{equation}\label{Wpkroneckerestimate}
W_p (\mu_N, \mathrm{Vol}) \ll p^{1/2} t^{1/2} + \frac{p(1-\delta^{1/p})}{(1-\delta)N} \bigg( \sum_{\substack{k \in \mathbb{Z}^d \\ k \neq 0}} \frac{e^{-4 \pi^2 |k|^2 q t}}{|k|^q \prod_{j=1}^m \| \langle k, \alpha_j \rangle \|^q_{\mathbb{R}/\mathbb{Z}}} \bigg)^{1/q} ,
\end{equation}
where
\[ \delta = \min \left\{ b \frac{e^{-r^2/(4t)}}{(4 \pi t)^{d/2}}, 1 \right\} . \]

Lemma \ref{cassels_lemma} states that any axis-parallel cube in $\mathbb{R}^d / \mathbb{Z}^d$ of sidelength $\asymp 1/\psi^{-1}(N)$ contains at least one point $n_1 \alpha_1+n_2 \alpha_2 + \cdots + n_m \alpha_m+\mathbb{Z}^d$ with $1 \le n_j \le N_j$, $1 \le j \le m$. Any ball $B(x,r) \subseteq \mathbb{R}^d / \mathbb{Z}^d$ contains $\gg (r \psi^{-1}(N))^d$ disjoint axis-parallel cubes of sidelength $\asymp 1/\psi^{-1}(N)$, therefore $\mu_N(B(x,r)) \gg N^{-1} (r \psi^{-1}(N))^d$. We can thus choose $b \asymp N^{-1} (r \psi^{-1}(N))^d$, hence
\[ \delta \gg \frac{1}{N} \left( r \psi^{-1}(N) \right)^d \frac{e^{-r^2/(4t)}}{(4 \pi t)^{d/2}} .  \]
The optimal choice is $r=t^{1/2}$, in which case $\delta \gg (\psi^{-1}(N))^d/N$. Observe
\[ \frac{p(1-\delta^{1/p})}{1-\delta} \le p, \]
and using $e^x \ge 1+x$, $x \in \mathbb{R}$, also
\[ \frac{p(1-\delta^{1/p})}{1-\delta} = \frac{p (1-e^{-(\log (1/\delta))/p})}{1-\delta} \le \frac{\log (1/\delta)}{1-\delta} . \]
Therefore
\begin{equation}\label{deltaestimate}
\frac{p (1-\delta^{1/p})}{1-\delta} \le \min \left\{ p, \frac{\log (1/\delta)}{1-\delta} \right\} \ll \min \left\{ p, \log \left(\frac{N}{(\psi^{-1}(N))^d} +2 \right) \right\} .
\end{equation}

In order to estimate the sum over $k$ in \eqref{Wpkroneckerestimate}, we use a lattice point counting estimate of Fregoli \cite[Proposition 1.2]{FR}, which states that under assumption \eqref{diophantinelattice}, for any $\varepsilon>0$ and $Q \ge 1$, we have
\[ \bigg| \bigg\{ k \in \mathbb{Z}^d : \prod_{j=1}^m \left\| \langle k, \alpha_j \rangle \right\|_{\mathbb{R} / \mathbb{Z}} < \varepsilon, \,\, 0<\| k \|_{\infty} \le Q \bigg\} \bigg| \ll \left( (\varepsilon \psi(Q))^{\frac{d+m-1}{d+m}} + \varepsilon Q^d \right) \left( \log \left( \frac{1}{\varepsilon} +2 \right) \right)^{m-1} . \]
In particular,
\[ \begin{split} \sum_{\substack{k \in \mathbb{Z}^d \\ 0<\| k \|_{\infty} \le Q}} &\frac{1}{\prod_{j=1}^m \| \langle k, \alpha_j \rangle \|_{\mathbb{R}/\mathbb{Z}}^q} \\ &\ll \sum_{1 \le n \le \frac{\log (\psi(Q)/I)}{\log 2}} 2^{qn} \bigg| \bigg\{ k \in \mathbb{Z}^d : \prod_{j=1}^m \left\| \langle k, \alpha_j \rangle \right\|_{\mathbb{R} / \mathbb{Z}} \in (2^{-n-1}, 2^{-n}], \,\, 0< \| k \|_{\infty} \le Q \bigg\} \bigg| \\ &\ll \sum_{1 \le n \le \frac{\log (\psi(Q)/I)}{\log 2}} \left( 2^{\left( q-\frac{d+m-1}{d+m} \right) n} \psi(Q)^{\frac{d+m-1}{d+m}} + 2^{(q-1)n} Q^d \right) n^{m-1} \\ &\ll \left( \psi(Q)^q + \frac{Q^d \psi(Q)^{q-1}}{2^{q-1}-1} \right) \left( \log (\psi(Q)+2) \right)^{m-1} . \end{split} \]
Here $1/(2^{q-1}-1) \asymp 1/(q-1) \asymp p$. For any integer $\ell \ge 0$, we thus have
\[ \begin{split} \sum_{\substack{k \in \mathbb{Z}^d \\ 2^{\ell} \le |k| < 2^{\ell+1}}} &\frac{e^{-4 \pi^2 |k|^2 q t}}{|k|^q \prod_{j=1}^m \| \langle k, \alpha_j \rangle \|^q_{\mathbb{R}/\mathbb{Z}}} \\ &\le \frac{e^{-4 \pi^2 2^{2\ell}qt}}{2^{q\ell}} \sum_{\substack{k \in \mathbb{Z}^d \\ 0< \|k \|_{\infty} < 2^{\ell+1}}} \frac{1}{\prod_{j=1}^m \| \langle k, \alpha \rangle \|^q_{\mathbb{R}/\mathbb{Z}}} \\ &\ll e^{-4 \pi^2 2^{2\ell}qt} \left( \frac{\psi (2^{\ell+1})^q}{2^{q \ell}}+p 2^{(d-q) \ell} \psi (2^{\ell+1})^{q-1} \right) (\log (\psi (2^{\ell+1})+2))^{m-1} . \end{split} \]
Summing the previous estimate over all $\ell \ge 0$ leads to
\begin{equation}\label{sumoverell1}
\begin{split} \sum_{\substack{k \in \mathbb{Z}^d \\ k \neq 0}} &\frac{e^{-4 \pi^2 |k|^2 q t}}{|k|^q \prod_{j=1}^m \| \langle k, \alpha_j \rangle \|^q_{\mathbb{R}/\mathbb{Z}}} \\ &\ll \sum_{\ell=0}^{\infty} e^{-4 \pi^2 2^{2\ell}qt} \left( \frac{\psi (2^{\ell+1})^q}{2^{q \ell}}+p 2^{(d-q) \ell} \psi (2^{\ell+1})^{q-1} \right) (\log (\psi (2^{\ell+1})+2))^{m-1}. \end{split}
\end{equation}
The assumption that $x^{-d}\psi(x)$ is nondecreasing shows that the terms with $2^{\ell+1} \le t^{-1/2}$ satisfy
\[ \psi(2^{\ell+1}) \le 2^{d(\ell+1)} t^{d/2} \psi (t^{-1/2}), \]
hence
\begin{equation}\label{sumoverell2}
\begin{split} \sum_{2^{\ell+1} \le t^{-1/2}} &e^{-4 \pi^2 2^{2\ell}qt} \left( \frac{\psi (2^{\ell+1})^q}{2^{q \ell}}+p 2^{(d-q) \ell} \psi (2^{\ell+1})^{q-1} \right) (\log (\psi (2^{\ell+1})+2))^{m-1} \\ &\ll \left( t^{dq/2} \psi (t^{-1/2})^q + p t^{d(q-1)/2} \psi(t^{-1/2})^{q-1} \right) (\log (\psi (t^{-1/2})+2))^{m-1} \sum_{2^{\ell+1} \le t^{-1/2}} 2^{(d-1)q \ell} \\ &\ll \left( t^{q/2} \psi (t^{-1/2})^q + p t^{(q-d)/2} \psi (t^{-1/2})^{q-1} \right) (\log (\psi (t^{-1/2})+2))^{m-1} . \end{split}
\end{equation}
The assumption $\psi(xy) \le c e^{x^2} \psi (y)$ shows that the terms with $2^{\ell+1} > t^{-1/2}$ satisfy
\[ \psi (2^{\ell+1}) \le c e^{2^{2\ell+2}t} \psi (t^{-1/2}), \]
hence
\begin{equation}\label{sumoverell3}
\begin{split} \sum_{2^{\ell+1}>t^{-1/2}} & e^{-4 \pi^2 2^{2\ell}qt} \left( \frac{\psi (2^{\ell+1})^q}{2^{q \ell}}+p 2^{(d-q) \ell} \psi (2^{\ell+1})^{q-1} \right) (\log (\psi (2^{\ell+1})+2))^{m-1} \\ &\ll  \sum_{2^{\ell+1}>t^{-1/2}} e^{-4 (\pi^2-1) 2^{2\ell}qt} \left( \frac{\psi(t^{-1/2})^q}{2^{q \ell}} + p 2^{(d-q)\ell} \psi(t^{-1/2})^{q-1} \right) \\ &\hspace{20mm} \times \left( \log (\psi (t^{-1/2})+2) + 2^{2\ell}t \right)^{m-1} \\ &\ll \left( t^{q/2} \psi (t^{-1/2})^q + p t^{(q-d)/2} \psi (t^{-1/2})^{q-1} \right) (\log (\psi (t^{-1/2})+2))^{m-1} . \end{split}
\end{equation}
The estimates \eqref{sumoverell1}--\eqref{sumoverell3} yield
\[ \sum_{\substack{k \in \mathbb{Z}^d \\ k \neq 0}} \frac{e^{-4 \pi^2 |k|^2 q t}}{|k|^q \prod_{j=1}^m \| \langle k, \alpha_j \rangle \|^q_{\mathbb{R}/\mathbb{Z}}} \ll t^{q/2} \psi (t^{-1/2})^q \left( 1+p \frac{t^{-d/2}}{\psi (t^{-1/2})} \right) (\log (\psi (t^{-1/2})+2))^{m-1} . \]

Substituting the previous formula and \eqref{deltaestimate} in \eqref{Wpkroneckerestimate} leads to
\[ \begin{split} W_p (\mu_N, \mathrm{Vol}) &\ll p^{1/2} t^{1/2} + \min \left\{ p, \log \left( \frac{N}{(\psi^{-1}(N))^d}+2 \right) \right\} \\ &\hspace{24mm} \times \frac{t^{1/2} \psi (t^{-1/2})}{N} \left( 1+p \frac{t^{-d/2}}{\psi(t^{-1/2})} \right) (\log (\psi (t^{-1/2})+2))^{(m-1)/q} \\ &= p^{1/2} t^{1/2} + \min \left\{ p, \log \left( \frac{N}{(\psi^{-1}(N))^d}+2 \right) \right\} \frac{t^{1/2}}{N} F_p (t^{-1/2}) \end{split} \]
in terms of the function
\[ F_p(x) = \left( \psi(x)+px^d \right) \left( \log (\psi(x)+2) \right)^{(m-1)(p-1)/p} . \]
The optimal choice for $t$ is
\[ t \asymp \left( F_p^{-1} \left( \frac{p^{1/2} N}{\min \left\{ p, \log \left( \frac{N}{(\psi^{-1}(N))^d}+2 \right) \right\}} \right) \right)^{-2}, \]
which proves the desired estimate for $W_p (\mu_N, \mathrm{Vol})$.
\end{proof}

For the sake of completeness, we prove that a.e.\ system of vectors $\alpha_1, \alpha_2, \ldots, \alpha_m \in \mathbb{R}^d$ satisfies the Diophantine assumption \eqref{diophantinelattice} with $\psi (x)=x^d (\log x)^m (\log \log x)^{1+\varepsilon}$,  $\varepsilon>0$.
\begin{lem}\label{metriclemma} Let $\psi: [1, \infty) \to [1,\infty)$ and assume
\[ \sum_{n=1}^{\infty} \frac{(\log (\psi (n)+2))^{m-1}}{\psi(n)} n^{d-1}  < \infty . \]
Then a.e.\ system of vectors $\alpha_1, \alpha_2, \ldots, \alpha_m \in \mathbb{R}^d$ satisfies \eqref{diophantinelattice}.
\end{lem}

\begin{proof} The Diophantine condition \eqref{diophantinelattice} depends only on the fractional parts of the coordinates of the vectors $\alpha_1, \alpha_2, \ldots, \alpha_m$. In particular, it will be enough to prove that $\mathrm{Vol}$-a.e.\ $(\alpha_1, \alpha_2, \ldots, \alpha_m) \in \mathbb{R}^{dm}/\mathbb{Z}^{dm}$ satisfies \eqref{diophantinelattice}.

Let $k \in \mathbb{Z}^d$, $k \neq 0$. Observe that the map $\mathbb{R}^{dm} / \mathbb{Z}^{dm} \to \mathbb{R}^m / \mathbb{Z}^m$, $(\alpha_1, \alpha_2, \ldots, \alpha_m) \mapsto (\langle k, \alpha_1 \rangle, \langle k, \alpha_2 \rangle, \ldots, \langle k, \alpha_m \rangle)$ is a continuous surjective homomorphism. In particular, the map preserves the normalized Haar measure $\mathrm{Vol}$, hence for any $t>0$,
\[ \begin{split} \mathrm{Vol} \bigg( \bigg\{ (\alpha_1, \alpha_2, \ldots, \alpha_m) &\in \mathbb{R}^{dm} / \mathbb{Z}^{dm} : \prod_{j=1}^m \left\| \langle k, \alpha_j \rangle \right\|_{\mathbb{R}/\mathbb{Z}} \le t \bigg\} \bigg) \\ &= \mathrm{Vol} \bigg( \bigg\{ (x_1, x_2, \ldots, x_m) \in \mathbb{R}^m / \mathbb{Z}^m : \prod_{j=1}^m \| x_j \|_{\mathbb{R}/\mathbb{Z}} \le t \bigg\} \bigg) . \end{split} \]
Identifying $\mathbb{R}^m/\mathbb{Z}^m$ by the unit cube $(-1/2,1/2]^m$, one readily checks by induction on $m$ that
\[ \mathrm{Vol} \bigg( \bigg\{ (x_1, x_2, \ldots, x_m) \in \left( -\frac{1}{2}, \frac{1}{2} \right]^m : \prod_{j=1}^m |x_j| \le t \bigg\} \bigg) \ll t \left( \log \left( \frac{1}{t} +2 \right) \right)^{m-1} \]
with an implied constant that depends only on $m$. An application of the previous two formulas with $t=1/\psi (\| k \|_{\infty})$ yields
\[ \mathrm{Vol} \bigg( \bigg\{ (\alpha_1, \alpha_2, \ldots, \alpha_m) \in \mathbb{R}^{dm} / \mathbb{Z}^{dm} : \psi (\| k \|_{\infty}) \prod_{j=1}^m \left\| \langle k, \alpha_j \rangle \right\|_{\mathbb{R}/\mathbb{Z}} \le 1 \bigg\} \bigg) \ll \frac{(\log (\psi (\| k \|_{\infty})+2))^{m-1}}{\psi(\| k \|_{\infty})} . \]
The number of lattice points $k \in \mathbb{Z}^d$ such that $\| k \|_{\infty}=n$ is $\ll n^{d-1}$, therefore the sum of volumes converges by assumption:
\[ \sum_{\substack{k \in \mathbb{Z}^d \\ k \neq 0}} \frac{( \log (\psi(\| k \|_{\infty})+2))^{m-1}}{\psi(\| k \|_{\infty})} \ll \sum_{n=1}^{\infty} \frac{(\log (\psi(n)+2))^{m-1}}{\psi(n)} n^{d-1} < \infty . \]
The Borel--Cantelli lemma thus shows that for $\mathrm{Vol}$-a.e.\ $(\alpha_1, \alpha_2, \ldots, \alpha_m) \in \mathbb{R}^{dm} / \mathbb{Z}^{dm}$, the inequality
\[ \psi (\| k \|_{\infty}) \prod_{j=1}^m \| \langle k, \alpha_j \rangle \|_{\mathbb{R}/\mathbb{Z}} > 1 \]
holds for all but finitely many $k \in \mathbb{Z}^d$. In particular, \eqref{diophantinelattice} is satisfied.
\end{proof}

\subsection{Modular hyperbolas}

\begin{proof}[Proof of Theorem \ref{modularhyperbola_theorem}] Let $q$ be a prime and let $a \in (\mathbb{Z}/q\mathbb{Z})^{\times}$. The Fourier transform of the empirical measure $\mu_{a,q,d}$ is
\[ \widehat{\mu_{a,q,d}}(k) = \frac{1}{(q-1)^{d-1}} \sum_{\substack{x_1, x_2, \ldots, x_d \in (\mathbb{Z}/q\mathbb{Z})^{\times} \\ x_1 x_2 \cdots x_d \equiv a \pmod{q}}} e^{-2 \pi i (k_1 x_1 + k_2 x_2 + \cdots + k_d x_d)/q} , \]
which is clearly invariant under permutations of the coordinates of $k=(k_1, k_2, \ldots, k_d) \in \mathbb{Z}^d$.

If none of the integers $k_1, k_2, \ldots, k_d$ are divisible by $q$, then a reindexing of the sum leads to $\widehat{\mu_{a,q,d}}(k) = (q-1)^{-(d-1)} K_d((-1)^d a k_1 k_2 \cdots k_d;q)$, where
\[ K_d(n;q) = \sum_{\substack{x_1, x_2, \ldots, x_d \in (\mathbb{Z}/q\mathbb{Z})^{\times} \\ x_1 x_2 \cdots x_d \equiv n \pmod{q}}} e^{2 \pi i (x_1+x_2+\cdots+x_d)/q} \]
is the generalized Kloosterman sum. The Weil bound due to Deligne \cite[Section 6.7]{DE} states that $|K_d(n;q)| \le d q^{(d-1)/2}$ for all $n \in (\mathbb{Z}/q\mathbb{Z})^{\times}$, hence $|\widehat{\mu_{a,q,d}}(k)| \ll q^{-(d-1)/2}$.

If $j \ge 1$ many of the integers $k_1, k_2, \ldots, k_d$ are divisible by $q$, then
\[ \begin{split} \widehat{\mu_{a,q,d}}(k) &= \frac{1}{(q-1)^{d-1}} \sum_{x_{j+1}, x_{j+2}, \ldots, x_d \in (\mathbb{Z}/q\mathbb{Z})^{\times}} e^{-2 \pi i (k_{j+1} x_{j+1} + k_{j+2} x_{j+2} + \cdots + k_d x_d)/q} \sum_{\substack{x_1, x_2, \ldots, x_j \in (\mathbb{Z}/q \mathbb{Z})^{\times} \\ x_1 x_2 \cdots x_d \equiv a \pmod{q}}} 1 \\ &= \frac{1}{(q-1)^{d-j}} \left( \sum_{x_{j+1} \in (\mathbb{Z}/q\mathbb{Z})^{\times}} e^{-2 \pi i k_{j+1} x_{j+1}/q} \right) \cdots \left( \sum_{x_d \in (\mathbb{Z} /q \mathbb{Z})^{\times}} e^{-2 \pi i k_d x_d/q} \right) \\ &= \frac{(-1)^{d-j}}{(q-1)^{d-j}} . \end{split} \]

Let $1 \le H<q$ be an integer to be chosen. An application of Theorem \ref{torussmoothing_theorem1} with $p=2$ yields
\[ W_2 (\mu_{a,q,d}, \mathrm{Vol}) \ll \frac{1}{H} + \bigg( \sum_{\substack{k \in [-H,H]^d \\ k \neq 0}} \frac{|\widehat{\mu_{a,q,d}}(k)|^2}{|k|^2} \bigg)^{1/2} . \]
The contribution of all $k \in [-H,H]^d$ none of whose coordinates are zero is
\[ \sum_{k \in ([-H,H]\backslash \{ 0 \})^d} \frac{|\widehat{\mu_{a,q,d}}(k)|^2}{|k|^2} \ll \frac{1}{q^{d-1}} \sum_{\substack{k \in [-H,H]^d \\ k \neq 0}} \frac{1}{|k|^2} \ll \left\{ \begin{array}{ll} q^{-1} \log H & \textrm{if } d=2, \\ q^{-(d-1)} H^{d-2} & \textrm{if } d \ge 3. \end{array} \right. \]
The contribution of all $k \in [-H,H]^d$ with $1 \le j \le d-1$ many zero coordinates is
\[ \sum_{\substack{k \in [-H,H]^d \\ |\{ 1 \le i \le d \, : \, k_i=0 \}|=j}} \frac{|\widehat{\mu_{a,q,d}}(k)|^2}{|k|^2} \ll \frac{1}{q^{2(d-j)}} \sum_{\substack{\ell \in [-H,H]^{d-j} \\ \ell \neq 0}} \frac{1}{|\ell|^2} \ll \left\{ \begin{array}{ll} q^{-2(d-j)} H^{d-j-2} & \textrm{if } 1 \le j \le d-3, \\ q^{-2(d-j)} \log H & \textrm{if } j=d-2, \\ q^{-2(d-j)} & \textrm{if } j=d-1. \end{array} \right. \]
The assumption $H<q$ implies that the case $j=d-1$ yields the largest error term. Adding the previous two estimates leads to
\[ \sum_{\substack{k \in [-H,H]^d \\ k \neq 0}} \frac{|\widehat{\mu_{a,q,d}}(k)|^2}{|k|^2} \ll \left\{ \begin{array}{ll} q^{-1} \log H & \textrm{if } d=2, \\ q^{-(d-1)} H^{d-2} + q^{-2} & \textrm{if } d \ge 3, \end{array} \right. \]
consequently
\[ W_2 (\mu_{a,q,d}, \mathrm{Vol}) \ll \frac{1}{H} + \left\{ \begin{array}{ll} q^{-1/2} (\log H)^{1/2} & \textrm{if } d=2, \\ q^{-(d-1)/2} H^{(d-2)/2} + q^{-1} & \textrm{if } d \ge 3.  \end{array} \right. \]
The optimal choice $H \asymp q^{(d-1)/d}$ proves the desired estimate for the quadratic Wasserstein metric
\[ W_2 (\mu_{a,q,d}, \mathrm{Vol}) \ll \left\{ \begin{array}{ll} (\log q)^{1/2} q^{-1/2} & \textrm{if } d=2, \\ q^{-(d-1)/d} & \textrm{if } d \ge 3. \end{array} \right. \]
Note this matches the best possible rate \eqref{Wpbestrate} in dimension $d \ge 3$.

We use Lemma \ref{D2torus_lemma} in order to estimate $D_2(\mu_{a,q,d}, \mathrm{Vol})$. If all coordinates of $k \in \mathbb{Z}^d$ are divisible by $q$, then $|\widehat{\mu_{a,q,d}}(k)|=1$. The contribution of all such terms is
\[ \sum_{\substack{k \in \mathbb{Z}^d \\ k \neq 0 \\ q \mid k_1, k_2, \ldots, k_d}} \frac{|\widehat{\mu_{a,q,d}}(k)|^2}{|k|^{d+1}} = \sum_{\substack{\ell \in \mathbb{Z}^d \\ \ell \neq 0}} \frac{1}{q^{d+1} |\ell|^{d+1}} \ll \frac{1}{q^{d+1}} . \]
If at least one coordinate of $k$ is not divisible by $q$, then $|\widehat{\mu_{a,q,d}}(k)| \ll \max \{ q^{-1}, q^{-(d-1)/2} \}$. The contribution of all such terms is
\[ \sum_{\substack{k \in \mathbb{Z}^d \\ q \nmid k_i \textrm{ for some } i}} \frac{|\widehat{\mu_{a,q,d}}(k)|^2}{|k|^{d+1}} \ll \max \{ q^{-2}, q^{-(d-1)} \} \sum_{\substack{k \in \mathbb{Z}^d \\ k \neq 0}} \frac{1}{|k|^{d+1}} \ll \left\{ \begin{array}{ll} q^{-1} & \textrm{if } d=2, \\ q^{-2} & \textrm{if } d \ge 3. \end{array} \right. \]
This proves the upper bound
\[ D_2 (\mu_{a,q,d}, \mathrm{Vol}) \asymp \bigg( \sum_{\substack{k \in \mathbb{Z}^d \\ k \neq 0}} \frac{|\widehat{\mu_{a,q,d}}|^2}{|k|^{d+1}} \bigg)^{1/2} \ll \left\{ \begin{array}{ll} q^{-1/2} & \textrm{if } d=2, \\ q^{-1} & \textrm{if } d \ge 3. \end{array} \right. \]
Since $|\widehat{\mu_{a,q,d}}(k)| = 1/(q-1)$ whenever $k$ has exactly one coordinate that is not divisible by $q$, we have
\[ \sum_{\substack{k \in \mathbb{Z}^d \\ k \neq 0}} \frac{|\widehat{\mu_{a,q,d}}(k)|^2}{|k|^{d+1}} \ge \frac{1}{(q-1)^2}\sum_{\substack{\ell \in \mathbb{Z} \\ q \nmid \ell}} \frac{1}{|\ell|^{d+1}} \gg \frac{1}{q^2} . \]
This proves the matching lower bound $D_2 (\mu_{a,q,d}, \mathrm{Vol}) \gg q^{-1}$ in dimension $d \ge 3$.
\end{proof}

\begin{proof}[Proof of Theorem \ref{modularhyperbolaW2_theorem}] Let $q$ be a prime and let $a \in (\mathbb{Z}/q\mathbb{Z})^{\times}$. We first apply a smoothing procedure to $\mu_{a,q,2}$ using the heat kernel
\[ P_t (x) = \sum_{k \in \mathbb{Z}^2} e^{-4 \pi^2 |k|^2 t} e^{2 \pi i \langle k, x \rangle}, \qquad t>0, \,\, x \in \mathbb{R}^2 / \mathbb{Z}^2 . \]
The convolution $P_t * \mu_{a,q,2}$ is an absolutely continuous probability measure on $\mathbb{R}^2 / \mathbb{Z}^2$ and has density
\[ \frac{\mathrm{d}(P_t * \mu_{a,q,d})}{\mathrm{d}\mathrm{Vol}}(x) = \int_{\mathbb{R}^2 / \mathbb{Z}^2} P_t (x-y) \, \mathrm{d} \mu_{a,q,2}(y) = \sum_{k \in \mathbb{Z}^2} e^{-4 \pi^2 |k|^2 t} \widehat{\mu_{a,q,2}}(k) e^{2 \pi i \langle k, x \rangle} = 1+f(x) \]
with the smooth function
\[ f(x) = \sum_{\substack{k \in \mathbb{Z}^2 \\ k \neq 0}} e^{-4 \pi^2 |k|^2 t} \widehat{\mu_{a,q,2}}(k) e^{2 \pi i \langle k, x \rangle} . \]
The heat kernel is contractive with respect to the metric $W_2$, that is, $W_2 (P_t * \mu, P_t * \nu) \le W_2 (\mu, \nu)$ for any $\mu, \nu \in \mathcal{P}(\mathbb{R}^2 / \mathbb{Z}^2)$, see \cite{AST}. Since $P_t*\mathrm{Vol}=\mathrm{Vol}$, we have
\begin{equation}\label{W2heatkernellowerbound1}
W_2 (P_t * \mu_{a,q,2}, \mathrm{Vol}) \le W_2 (\mu_{a,q,2}, \mathrm{Vol}).
\end{equation}

As observed in the proof of Theorem \ref{modularhyperbola_theorem}, we have
\begin{equation}\label{muaq2fouriercoefficients}
\widehat{\mu_{a,q,2}}(k) = \left\{ \begin{array}{ll} 1 & \textrm{if } q \mid k_1, k_2, \\ -1/(q-1) & \textrm{if } q \textrm{ divides exactly one of } k_1, k_2, \\ K_2 (ak_1 k_2;q)/(q-1) & \textrm{if } q \nmid k_1, k_2,  \end{array} \right.
\end{equation}
where
\[ K_2 (n;q) = \sum_{x \in (\mathbb{Z}/q\mathbb{Z})^{\times}} e^{2 \pi i (nx+\overline{x})/q} \]
is the Kloosterman sum. The contribution of the terms $q \mid k_1, k_2$ in $f(x)$ is
\[ \Bigg| \sum_{\substack{k \in \mathbb{Z}^2 \\ k \neq 0 \\ q \mid k_1, k_2}} e^{-4 \pi^2 |k|^2 t} \widehat{\mu_{a,q,2}}(k) e^{2 \pi i \langle k, x \rangle} \Bigg| \le \sum_{\substack{\ell \in \mathbb{Z}^2 \\ \ell \neq 0}} e^{-4 \pi^2 q^2 |\ell|^2 t} \ll \frac{1}{q^2 t} . \]
Using the Weil bound $|K_2 (n;q)| \le 2 q^{1/2}$, $n \in (\mathbb{Z}/q\mathbb{Z})^{\times}$, the contribution of all other terms is
\[ \Bigg |\sum_{\substack{k \in \mathbb{Z}^2 \\ k \neq 0 \\ q \nmid k_1 \textrm{ or } q \nmid k_2}} e^{-4 \pi^2 |k|^2 t} \widehat{\mu_{a,q,2}}(k) e^{2 \pi i \langle k, x \rangle} \Bigg| \ll \frac{1}{q^{1/2}} \sum_{\substack{k \in \mathbb{Z}^2 \\ k \neq 0}} e^{-4 \pi^2 |k|^2 t} \ll \frac{1}{q^{1/2} t} . \]
The previous two estimates show that $|f(x)| \ll q^{-1/2} t^{-1}$. Choosing, say, $t=q^{-1/4}$, we thus have $|f(x)| \ll q^{-1/4}$, and an application of Lemma \ref{ledoux_lemma} leads to
\begin{equation}\label{W2heatkernellowerbound2}
\left( \int_{\mathbb{R}^2 / \mathbb{Z}^2} f (\Delta^{-1} f) \, \mathrm{d} \mathrm{Vol} \right)^{1/2} \ll W_2 (P_t * \mu_{a,q,2}, \mathrm{Vol}) .
\end{equation}

Applying the inverse of the Laplace--Beltrami operator to $f(x)$ gives the absolutely convergent Fourier series
\[ (\Delta^{-1} f)(x) = \sum_{\substack{k \in \mathbb{Z}^2 \\ k \neq 0}} e^{-4 \pi^2 |k|^2 t} \frac{\widehat{\mu_{a,q,2}}(k)}{4 \pi^2 |k|^2} e^{2 \pi i \langle k, x \rangle}, \]
therefore the Parseval formula leads to
\[ \int_{\mathbb{R}^2 / \mathbb{Z}^2} f (\Delta^{-1} f) \, \mathrm{d} \mathrm{Vol} = \sum_{\substack{k \in \mathbb{Z}^2 \\ k \neq 0}} e^{-8 \pi^2 |k|^2 t} \frac{|\widehat{\mu_{a,q,2}}(k)|^2}{4 \pi^2 |k|^2} . \]
Estimates \eqref{W2heatkernellowerbound1}, \eqref{W2heatkernellowerbound2} and the previous formula yield
\[ W_2 (\mu_{a,q,2}, \mathrm{Vol})^2 \gg \sum_{\substack{k \in \mathbb{Z}^2 \\ k \neq 0}} e^{-8 \pi^2 |k|^2 t} \frac{|\widehat{\mu_{a,q,2}}(k)|^2}{|k|^2} \gg \frac{1}{q^2} \sum_{\substack{k \in \mathbb{Z}^2 \\ q \nmid k_1, k_2}} e^{-8 \pi^2 |k|^2 t} \frac{K_2(ak_1 k_2; q)^2}{|k|^2} . \]

Let us now take the average over $a \in (\mathbb{Z}/q\mathbb{Z})^{\times}$. Observe that for any given $k_1, k_2$ not divisible by $q$, the values $ak_1k_2$ attain each residue class in $(\mathbb{Z}/q\mathbb{Z})^{\times}$ exactly once. Therefore
\[ \frac{1}{q-1} \sum_{a \in (\mathbb{Z}/q\mathbb{Z})^{\times}} W_2 (\mu_{a,q,2}, \mathrm{Vol})^2 \gg \frac{1}{q^3} \sum_{a \in (\mathbb{Z}/q\mathbb{Z})^{\times}} K_2 (a;q)^2 \sum_{\substack{k \in \mathbb{Z}^2 \\ q \nmid k_1, k_2}} \frac{e^{-8 \pi^2 |k|^2 t}}{|k|^2} . \]
The sum of squares of Kloosterman sums is evaluated as
\begin{equation}\label{kloostermansecondmoment}
\begin{split} \sum_{a \in (\mathbb{Z}/q\mathbb{Z})^{\times}} K_2(a;q)^2 &= \sum_{x \in (\mathbb{Z}/q\mathbb{Z})^{\times}} \sum_{y \in (\mathbb{Z}/q\mathbb{Z})^{\times}} \sum_{a \in (\mathbb{Z}/q\mathbb{Z})^{\times}} e^{2 \pi i (a(x-y) +\overline{x}-\overline{y})/q} \\ &= \sum_{x \in (\mathbb{Z}/q\mathbb{Z})^{\times}} \sum_{y \in (\mathbb{Z}/q\mathbb{Z})^{\times}} \left( (q-1) \mathds{1}_{\{ x=y \}} - e^{2 \pi i (\overline{x}-\overline{y})/q} \mathds{1}_{\{ x \neq y \}} \right) \\ &= \sum_{x \in (\mathbb{Z}/q\mathbb{Z})^{\times}} \left( q+e^{2 \pi i \overline{x}/q} \right) = q^2-q-1. \end{split}
\end{equation}
Finally, by the choice $t=q^{-1/4}$,
\[ \sum_{\substack{k \in \mathbb{Z}^2 \\ q \nmid k_1, k_2}} \frac{e^{-8 \pi^2 |k|^2 t}}{4 \pi^2 |k|^2} \gg \sum_{\substack{k \in \mathbb{Z}^2 \\ |k| \le q^{1/8} \\ q \nmid k_1, k_2}} \frac{1}{|k|^2} \gg \log q . \]
The previous three estimates prove the desired relation
\[ \frac{1}{q-1} \sum_{a \in (\mathbb{Z}/q\mathbb{Z})^{\times}} W_2 (\mu_{a,q,2}, \mathrm{Vol})^2 \gg \frac{\log q}{q}. \]
\end{proof}

\begin{proof}[Proof of Theorem \ref{modularhyperbolaD2_theorem}] Let $q$ be a prime, let $a \in (\mathbb{Z}/q\mathbb{Z})^{\times}$ and let $0<r \le 1/2$. Lemma \ref{D2torus_lemma} states that
\[ D_2(\mu_{a,q,2}, \mathrm{Vol},r)^2 = \sum_{\substack{k \in \mathbb{Z}^2 \\ k \neq 0}} b_k(r) |\widehat{\mu_{a,q,2}}(k)|^2 \]
with coefficients $b_k(r)=r^2 |k|^{-2} J_1 (2 \pi r |k|)^2 \ll \min \{ r^4, r |k|^{-3} \}$. The total sum of all coefficients is $\sum_{k \neq 0} b_k(r) = \pi r^2 (1-\pi r^2)$.

Recall \eqref{muaq2fouriercoefficients} and note that both $b_k(r)$ and $\widehat{\mu_{a,q,2}}(k)$ are invariant under switching the coordinates of $k$. Hence
\begin{equation}\label{D2muaq2explicit}
D_2 (\mu_{a,q,2}, \mathrm{Vol}, r)^2 = \frac{1}{(q-1)^2} \sum_{\substack{k \in \mathbb{Z}^2 \\ q \nmid k_1, k_2}} b_k(r) K_2(ak_1k_2;q)^2 + \frac{2}{(q-1)^2} \sum_{\substack{k \in \mathbb{Z}^2 \\ q \mid k_1, \,\, q \nmid k_2}} b_k(r) + \sum_{\substack{k \in \mathbb{Z}^2 \\ k \neq 0 \\ q \mid k_1, k_2}} b_k(r) .
\end{equation}
Taking the average over $a \in (\mathbb{Z}/q\mathbb{Z})^{\times}$ and using formula \eqref{kloostermansecondmoment} for the sum of squares of Kloosterman sums leads to
\begin{equation}\label{D2muaq2}
\begin{split} \frac{1}{q-1} \sum_{a \in (\mathbb{Z}/q\mathbb{Z})^{\times}} &D_2 (\mu_{a,q,2}, \mathrm{Vol}, r)^2 \\ &= \frac{q^2-q-1}{(q-1)^3} \sum_{\substack{k \in \mathbb{Z}^2 \\ q \nmid k_1, k_2}} b_k(r) + \frac{2}{(q-1)^2} \sum_{\substack{k \in \mathbb{Z}^2 \\ q \mid k_1, \,\, q \nmid k_2}} b_k(r) + \sum_{\substack{k \in \mathbb{Z}^2 \\ k \neq 0 \\ q \mid k_1, k_2}} b_k(r) . \end{split}
\end{equation}

It remains to estimate the sums of coefficients $b_k(r)$ in \eqref{D2muaq2}. Consider the last sum in \eqref{D2muaq2}. If $r \le 1/q$, then
\[ \sum_{\substack{k \in \mathbb{Z}^2 \\ k \neq 0 \\ q \mid k_1, k_2}} b_k(r) \ll \sum_{\substack{\ell \in \mathbb{Z}^2 \\ \ell \neq 0}} \min \left\{ r^4, \frac{r}{q^3 |\ell|^3} \right\} \ll \sum_{\substack{\ell \in \mathbb{Z}^2 \\ 0<|\ell| \le 1/(rq)}} r^4 + \sum_{\substack{\ell \in \mathbb{Z}^2 \\ |\ell| > 1/(rq)}} \frac{r}{q^3 |\ell|^3} \ll \frac{r^2}{q^2}. \]
If $r>1/q$, then
\[ \sum_{\substack{k \in \mathbb{Z}^2 \\ k \neq 0 \\ q \mid k_1, k_2}} b_k(r) \ll \sum_{\substack{\ell \in \mathbb{Z}^2 \\ \ell \neq 0}} \frac{r}{q^3 |\ell|^3} \ll \frac{r}{q^3}. \]
Altogether,
\begin{equation}\label{qmidk1k2}
\sum_{\substack{k \in \mathbb{Z}^2 \\ k \neq 0 \\ q \mid k_1, k_2}} b_k(r) \ll \min \left\{ \frac{r^2}{q^2}, \frac{r}{q^3} \right\} .
\end{equation}

Consider now the second sum in \eqref{D2muaq2} with the extra condition $k_1 \neq 0$. If $r \le 1/q$, then
\[ \begin{split} \sum_{\substack{k \in \mathbb{Z}^2 \\ q \mid k_1, \,\, q \nmid k_2 \\ k_1 \neq 0}} b_k(r) &\ll \sum_{\substack{\ell \in \mathbb{Z}^2 \\ \ell \neq 0}} \min \left\{ r^4, \frac{r}{(q^2 \ell_1^2 + \ell_2^2)^{3/2}} \right\} \\ &\ll \sum_{|\ell_1| \le 1/(rq)} \sum_{|\ell_2| \le 1/r} r^4 + \sum_{|\ell_1| \le 1/(rq)} \sum_{|\ell_2| > 1/r} \frac{r}{|\ell_2|^3} \\ &{\phantom{\ll}} + \sum_{|\ell_1| > 1/(rq)} \sum_{|\ell_2| \le q|\ell_1|} \frac{r}{q^3 |\ell_1|^3} + \sum_{|\ell_1| > 1/(rq)} \sum_{|\ell_2|> q |\ell_1|} \frac{r}{|\ell_2|^3} \\ &\ll  \frac{r^2}{q} . \end{split} \]
If $r>1/q$, then
\[ \sum_{\substack{k \in \mathbb{Z}^2 \\ q \mid k_1, \,\, q \nmid k_2 \\ k_1 \neq 0}} b_k(r) \ll \sum_{\substack{\ell \in \mathbb{Z}^2 \\ \ell_1, \ell_2 \neq 0}} \frac{r}{(q^2 \ell_1^2 + \ell_2^2)^{3/2}} \ll \sum_{\ell_1 \neq 0} \sum_{|\ell_2| \le q |\ell_1|} \frac{r}{q^3 |\ell_1|^3} + \sum_{\ell_1 \neq 0} \sum_{|\ell_2 | > q|\ell_1|} \frac{r}{|\ell_2|^3} \ll \frac{r}{q^2} . \]
Altogether,
\begin{equation}\label{qmidk1nmidk2}
\sum_{\substack{k \in \mathbb{Z}^2 \\ q \mid k_1, \,\, q \nmid k_2 \\ k_1 \neq 0}} b_k(r) \ll \min \left\{ \frac{r^2}{q}, \frac{r}{q^2} \right\} .
\end{equation}
Combining \eqref{D2muaq2}--\eqref{qmidk1nmidk2} leads to
\begin{equation}\label{D2muaq2second}
\begin{split} \frac{1}{q-1} \sum_{a \in (\mathbb{Z}/q\mathbb{Z})^{\times}} &D_2 (\mu_{a,q,2}, \mathrm{Vol}, r)^2 \\ &= \frac{q^2-q-1}{(q-1)^3} \sum_{\substack{k \in \mathbb{Z}^2 \\ q \nmid k_1, k_2}} b_k(r) + \frac{2}{(q-1)^2} \sum_{\substack{k \in \mathbb{Z}^2 \\ k_1=0, \,\, q \nmid k_2}} b_k(r)+ O \left( \min \left\{ \frac{r^2}{q^2}, \frac{r}{q^3} \right\} \right) . \end{split}
\end{equation}

The reason for the deviation of the constants from the i.i.d.\ case is that the sum of the coefficients $b_k(r)$ over $k_1=0$, $k_2 \neq 0$ is not negligible. Indeed,
\[ \sum_{\substack{k \in \mathbb{Z}^2 \\ k_1=0, \,\, q \nmid k_2}} b_k(r) = \sum_{\substack{\ell \in \mathbb{Z} \\ q \nmid \ell}} \frac{r^2}{\ell^2} J_1 (2 \pi r |\ell|)^2 = \sum_{\substack{\ell \in \mathbb{Z} \\ \ell \neq 0}} \frac{r^2}{\ell^2} J_1 (2 \pi r |\ell|)^2 + O \left( \frac{r^2}{q^2} \right) . \]
To evaluate the infinite series in the previous formula, observe that the Fourier coefficients of the even function $f(x)=\mathds{1}_{\{ |x| \le r \}} \sqrt{1-x^2/r^2}$, $x \in [-1/2, 1/2]$, are
\[ \int_{-1/2}^{1/2} f(x) e^{-2 \pi i \ell x} \, \mathrm{d}x = r \int_{-1}^1 \sqrt{1-x^2} e^{2 \pi i r |\ell| x} \, \mathrm{d}x = \frac{J_1(2 \pi r |\ell|)}{2 |\ell|} , \quad \ell \in \mathbb{Z} \backslash \{ 0 \} . \]
An application of the Parseval formula thus yields
\[ \sum_{\substack{\ell \in \mathbb{Z} \\ \ell \neq 0}} \frac{J_1 (2 \pi r |\ell|)^2}{4 \ell^2} = \int_{-r}^r \left( 1-\frac{x^2}{r^2} \right) \, \mathrm{d}x - \left( \int_{-r}^r \sqrt{1-\frac{x^2}{r^2}} \, \mathrm{d}x \right)^2 = \frac{4}{3} r - \frac{\pi^2}{4} r^2, \]
hence
\[ \sum_{\substack{k \in \mathbb{Z}^2 \\ k_1=0, \,\, q \nmid k_2}} b_k(r) = \frac{16}{3} r^3 - \pi^2 r^4 + O \left( \frac{r^2}{q^2} \right) . \]
The fact that $\sum_{k \neq 0} b_k(r)=\pi r^2 (1-\pi r^2)$, \eqref{qmidk1k2}, \eqref{qmidk1nmidk2} and the previous formula lead to
\[ \sum_{\substack{k \in \mathbb{Z}^2 \\ q \nmid k_1, k_2}} b_k(r) = \pi r^2 (1-\pi r^2) - \frac{32}{3} r^3 + 2 \pi^2 r^4 +O \left( \min \left\{ \frac{r^2}{q}, \frac{r}{q^2} \right\} \right) . \]
The previous two formulas simplify \eqref{D2muaq2second} to
\[ \begin{split} \frac{1}{q-1} &\sum_{a \in (\mathbb{Z}/q\mathbb{Z})^{\times}} D_2 (\mu_{a,q,2}, \mathrm{Vol}, r)^2 \\ &= \frac{q^2-q-1}{(q-1)^3} \left( \pi r^2 (1-\pi r^2) - \frac{32}{3} r^3 + 2 \pi^2 r^4 \right) + \frac{(32/3)r^3 - 2 \pi^2 r^4}{(q-1)^2} + O \left( \min \left\{ \frac{r^2}{q^2}, \frac{r}{q^3} \right\} \right) . \end{split} \]
Here $(q^2-q-1)/(q-1)^3 = 1/(q-1) + 1/(q-1)^2+O(q^{-3})$, and the desired formula for the quadratic average of $D_2 (\mu_{a,q,2}, \mathrm{Vol},r)$ follows. Further averaging over $0<r \le 1/2$ leads to the claim about the quadratic average of $D_2(\mu_{a,q,2}, \mathrm{Vol})$.
\end{proof}

As observed in the previous proof, the $L^2$ discrepancy with respect to balls of the modular hyperbola is closely related to the distribution of Kloosterman sums. The Weil bound $|K_2(a;q)| \le 2 \sqrt{q}$ means that we can write $K_2(a;q)=2 \sqrt{q} \cos (\psi_q(a))$ with a unique angle $\psi_q (a) \in [0, \pi]$. For a fixed prime $q$, the angles $\psi_q(a)$, $a \in (\mathbb{Z}/q\mathbb{Z})^{\times}$, are known to equidistribute in $[0,\pi]$ with respect to the Sato--Tate measure $(2/\pi) \sin^2 x \, \mathrm{d}x$. This is known as the vertical Sato--Tate conjecture for Kloosterman sums, and was proved by Katz \cite[Chapter 13]{KA2}. We were able to sidestep this deep result by relying on the explicit formula \eqref{kloostermansecondmoment} instead.

The horizontal Sato--Tate conjecture for Kloosterman sums, formulated by Katz \cite[Conjecture 1.2.5]{KA1}, predicts that for an arbitrary integer $a \neq 0$, the angles $\psi_q(a)$ as $q$ runs over all primes that do not divide $a$ up to a threshold also equidistribute in $[0,\pi]$ with respect to the Sato--Tate measure. That is, for any interval $[\alpha, \beta] \subseteq [0, \pi]$,
\begin{equation}\label{satotate}
\lim_{Q \to \infty} \frac{1}{\pi(Q)} \left| \left\{ q \le Q : q \textrm{ prime}, \,\, q \nmid a, \,\, \psi_q (a) \in [\alpha, \beta] \right\} \right| = \int_{\alpha}^{\beta} \frac{2}{\pi} \sin^2 x \, \mathrm{d}x.
\end{equation}

We now prove that the horizontal Sato--Tate conjecture \eqref{satotate} implies relation \eqref{averageD2overprimes}. Fix a nonzero integer $a$ and a large constant $C>0$. The explicit formula \eqref{D2muaq2explicit}, estimates \eqref{qmidk1k2}, \eqref{qmidk1nmidk2} and the calculations in the proof of Theorem \ref{modularhyperbolaD2_theorem} show
\[ (q-1) D_2(\mu_{a,q,2}, \mathrm{Vol},r)^2 = \sum_{\substack{k \in \mathbb{Z}^2 \\ q \nmid k_1, k_2}} 4 b_k(r) \cos^2 (\psi_q(ak_1k_2)) + O \left( \frac{1}{q} \right) \]
with a universal implied constant. Let us now take the average over all primes $q \le Q$ such that $q \nmid a$. Observing that
\[ \sum_{\substack{k \in \mathbb{Z}^2 \\ q \nmid k_1, k_2 \\ |k| \ge C}} 4 b_k(r) \cos^2 (\psi_q(ak_1k_2)) \ll \sum_{\substack{k \in \mathbb{Z}^2 \\ |k| \ge C}} \frac{1}{|k|^3} \ll \frac{1}{C} \]
and
\[ \frac{1}{\pi(Q)} \sum_{\substack{q \le C \\ q \textrm{ prime}, \,\, q \nmid a}} (q-1) D_2(\mu_{a,q,2}, \mathrm{Vol},r)^2 \ll \frac{\pi (C)}{\pi (Q)}, \]
we deduce
\[ \begin{split} \frac{1}{\pi(Q)} \sum_{\substack{q \le Q \\ q \textrm{ prime}, \,\, q \nmid a}} &(q-1) D_2(\mu_{a,q,2}, \mathrm{Vol},r)^2 \\ &= \frac{1}{\pi(Q)} \sum_{\substack{C< q \le Q \\ q \textrm{ prime}, \,\, q \nmid a}} \sum_{\substack{k \in \mathbb{Z}^2 \\ q \nmid k_1, k_2 \\ |k| < C}} 4 b_k(r) \cos^2 (\psi_q(ak_1 k_2)) + O \left( \frac{\log \log Q}{\pi(Q)}+ \frac{1}{C}+ \frac{\pi(C)}{\pi(Q)} \right) \\ &= \sum_{\substack{k \in \mathbb{Z}^2 \\ k_1, k_2 \neq 0 \\ |k|<C}} b_k(r) \frac{1}{\pi(Q)} \sum_{\substack{C<q \le Q \\ q \textrm{ prime}, \,\, q \nmid a}} 4 \cos^2 (\psi_q(ak_1k_2)) + O \left( \frac{\log \log Q}{\pi(Q)}+ \frac{1}{C}+ \frac{\pi(C)}{\pi(Q)} \right) \end{split} \]
with an implied constant that depends only on $a$. We used the fact that in all terms of the remaining sum, we have $|k_1|, |k_2|<C<q$, consequently the condition $q \nmid k_1, k_2$ is equivalent to $k_1, k_2 \neq 0$. An application of the horizontal Sato--Tate conjecture \eqref{satotate} shows that
\[ \lim_{Q \to \infty} \frac{1}{\pi(Q)} \sum_{\substack{C< q \le Q \\ q \textrm{ prime}, \,\, q \nmid a}} 4 \cos^2 (\psi_q(ak_1 k_2)) = \frac{2}{\pi} \int_0^{\pi} 4 \cos^2 x \sin^2 x \, \mathrm{d}x = 1 \]
uniformly in the finitely many nonzero integers $ak_1 k_2$, $|k| <C$. Hence
\[ \begin{split} \limsup_{Q \to \infty} \frac{1}{\pi(Q)} \sum_{\substack{q \le Q \\ q \textrm{ prime}, \,\, q \nmid a}} (q-1) D_2(\mu_{a,q,2}, \mathrm{Vol},r)^2 &= \sum_{\substack{k \in \mathbb{Z}^2 \\ k_1, k_2 \neq 0 \\ |k| < C}} b_k(r) + O \left( \frac{1}{C} \right) \\ &= \sum_{\substack{k \in \mathbb{Z}^2 \\ k \neq 0}} b_k(r) - 2 \sum_{\substack{k \in \mathbb{Z}^2 \\ k_1=0, \,\, k_2 \neq 0}} b_k(r) + O \left( \frac{1}{C} \right) \\ &= \pi r^2 - \frac{32}{3} r^3 + \pi^2 r^4 + O \left( \frac{1}{C} \right) . \end{split} \]
Letting $C \to \infty$ leads to
\[ \limsup_{Q \to \infty} \frac{1}{\pi(Q)} \sum_{\substack{q \le Q \\ q \textrm{ prime}, \,\, q \nmid a}} (q-1) D_2(\mu_{a,q,2}, \mathrm{Vol},r)^2 = \pi r^2 - \frac{32}{3} r^3 + \pi^2 r^4, \]
and the same holds for the liminf. This finishes the proof of \eqref{averageD2overprimes} under the horizontal Sato--Tate conjecture.

\subsection{Modular parabolas}

\begin{proof}[Proof of Theorem \ref{modularparabola_theorem}] Let $q$ be a prime and let $a \in (\mathbb{Z}/q\mathbb{Z})^{\times}$, $b \in \mathbb{Z}/q\mathbb{Z}$. The Fourier transform of the empirical measure $\mu_{a,b,q}$ is
\[ \widehat{\mu_{a,b,q}}(k) = \frac{1}{q} \sum_{x \in \mathbb{Z}/q\mathbb{Z}} e^{-2 \pi i (k_1 x + k_2 (ax^2+bx))/q} . \]
If $q \nmid k_2$, then this is a quadratic Gauss sum, therefore $|\widehat{\mu_{a,b,q}}(k)|=1/\sqrt{q}$, see \cite[p.\ 50]{IK}. If $q \mid k_2$, then the sum can be evaluated as a geometric series. Hence
\begin{equation}\label{gausssum}
|\widehat{\mu_{a,b,q}}(k)| = \left\{ \begin{array}{ll} 1 & \textrm{if } q \mid k_1, k_2, \\ 0 & \textrm{if } q \nmid k_1, \,\, q \mid k_2, \\ 1/\sqrt{q} & \textrm{if } q \nmid k_2 . \end{array} \right.
\end{equation}
Following the steps in the proof of Theorems \ref{modularhyperbola_theorem} and \ref{modularhyperbolaW2_theorem} establishes the claim $W_2 (\mu_{a,b,q}, \mathrm{Vol}) \asymp (\log q)^{1/2} q^{-1/2}$.

Using Lemma \ref{D2torus_lemma} and \eqref{gausssum} leads to
\[ \begin{split} D_2(\mu_{a,b,q}, \mathrm{Vol},r)^2 &= \sum_{\substack{k \in \mathbb{Z}^2 \\ k \neq 0}} b_k(r) |\widehat{\mu_{a,b,q}}(k)|^2 \\ &= \frac{1}{q} \sum_{\substack{k \in \mathbb{Z}^2 \\ q \nmid k_1, k_2}} b_k(r) + \frac{1}{q} \sum_{\substack{k \in \mathbb{Z}^2 \\ q \nmid k_2, \,\, k_1=0}} b_k(r) + \frac{1}{q} \sum_{\substack{k \in \mathbb{Z}^2 \\ q \mid k_1, \,\, q \nmid k_2 \\ k_1 \neq 0 }} b_k(r) + \sum_{\substack{k \in \mathbb{Z}^2 \\ q \mid k_1, k_2 \\ k \neq 0}} b_k(r) . \end{split} \]
All four sums in the previous formula were estimated in the proof of Theorem \ref{modularhyperbolaD2_theorem}, we thus deduce
\[ \begin{split} D_2(\mu_{a,b,q}, \mathrm{Vol}, r)^2 &= \frac{\pi r^2 (1-\pi r^2)-(32/3) r^3 + 2 \pi^2 r^4}{q} + \frac{(16/3) r^3 - \pi^2 r^4}{q} + O \left( \min \left\{ \frac{r^2}{q^2}, \frac{r}{q^3} \right\} \right) \\ &= \frac{\pi r^2 - (16/3)r^3}{q} + O \left( \min \left\{ \frac{r^2}{q^2}, \frac{r}{q^3} \right\} \right) . \end{split} \]
Further averaging over $0 \le r < 1/2$ yields
\[ D_2 (\mu_{a,b,q}, \mathrm{Vol})^2 = \frac{\pi / 12- 1/6}{q} + O \left( \frac{1}{q^3} \right) . \]
\end{proof}

\section{Equidistribution on the sphere}\label{sphere_section}

We give the proof of Theorem \ref{fibonacci_theorem} in this section. Let $\alpha \in \mathbb{R}$, $N \in \mathbb{N}$, and assume we have $\min_{1 \le n \le \sqrt{N}} n \| n \alpha \|_{\mathbb{R}/\mathbb{Z}} \ge c$ with a fixed constant $c>0$. We may assume $c \le 1/4$.

For technical reasons, we work in two topologically different versions of the unit square. First, the torus $\mathbb{R}^2 / \mathbb{Z}^2$ equipped with the metric
\[ \rho_{\mathbb{R}^d/\mathbb{Z}^d} (x,y) = \| x-y \|_{\mathbb{R}^2/\mathbb{Z}^2} = \sqrt{\| x_1 -y_1 \|_{\mathbb{R}/\mathbb{Z}}^2 + \| x_2 -y_2 \|_{\mathbb{R}/\mathbb{Z}}^2}. \]
Second, the cylinder $(\mathbb{R}/\mathbb{Z}) \times [0,1]$ equipped with the metric
\[ \rho_{(\mathbb{R}/\mathbb{Z})\times [0,1]} (x,y)=\sqrt{\| x_1 -y_1 \|_{\mathbb{R}/\mathbb{Z}}^2 + |x_2-y_2|^2} . \]
One can visualize both spaces as the unit square: in the torus both pairs of opposite sides are identified, whereas in the cylinder only the two vertical sides are identified. We consider the Lambert cylindrical equal area projection as a map defined on the cylinder, $\phi: (\mathbb{R}/\mathbb{Z}) \times [0,1] \to \mathbb{S}^2$,
\[ \phi (x_1, x_2) = \left( 2 \sqrt{x_2-x_2^2} \cos (2 \pi x_1), 2 \sqrt{x_2-x_2^2} \sin (2 \pi x_1), 1-2 x_2 \right) . \]
Note that the first variable $x_1$ appears only in $1$-periodic trigonometric functions, hence $\phi$ is continuous. On the other hand, $\phi$ maps the bottom side of the unit square $x_2=0$ to the north pole of the sphere, and the top side of the unit square $x_2=1$ to the south pole. In particular, $\phi$ would not be continuous as a map defined on the torus.

Consider the map
\[ f: \mathbb{R}^2 / \mathbb{Z}^2 \to (\mathbb{R}/\mathbb{Z}) \times [0,1], \quad f(x_1, x_2) = \left\{ \begin{array}{ll} (x_1, 2 \{ x_2 \}) & \textrm{if } 0 \le \{ x_2 \} < 1/2, \\ (x_1, 2-2\{ x_2 \}) & \textrm{if } 1/2 \le \{ x_2 \} <1, \end{array} \right. \]
where the fractional part $\{ x_2 \}$ is well-defined for $x_2 \in \mathbb{R}/\mathbb{Z}$. Geometrically, $f$ folds the unit square in half along the horizontal line $x_2=1/2$, then vertically stretches by a factor of 2 to obtain a square. The point is that the bottom and top horizontal sides are both mapped to the bottom side of the unit square, which makes $f$ continuous. One readily checks that $f$ is in fact Lipschitz with Lipschitz constant $2$, and preserves the Lebesgue measure $\mathrm{Vol}$.

The preimages of the point $(n \alpha+\mathbb{Z}, n/N) \in (\mathbb{R}/\mathbb{Z}) \times [0,1]$ with respect to $f$ are $(n \alpha +\mathbb{Z}, n/(2N))$ and $(n \alpha +\mathbb{Z}, 1-n/(2N))$ in $\mathbb{R}^2 / \mathbb{Z}^2$. In particular, the empirical measure
\[ \nu_N = \frac{1}{2N} \sum_{n=0}^{N-1} \left( \delta_{(n \alpha+\mathbb{Z}, n/(2N))} + \delta_{(n \alpha + \mathbb{Z}, 1-n/(2N))} \right) \in \mathcal{P} \left( \mathbb{R}^2 / \mathbb{Z}^2 \right) \]
has pushforward
\[ f_* \nu_N = \frac{1}{N} \sum_{n=0}^{N-1} \delta_{(n \alpha + \mathbb{Z}, n/N)} \in \mathcal{P} \left( (\mathbb{R}/\mathbb{Z}) \times [0,1] \right) . \]
We can thus express the empirical measure $\mu_N$ of the set $S_N(\alpha)$ as $\mu_N = \phi_* f_* \nu_N$. Our strategy is to estimate $W_p (\nu_N, \mathrm{Vol})$ in the torus using a smoothing inequality, then estimate the effect of the two pushforward operations.

\begin{lem}\label{fibonacci_lemma} For any $2 \le p < \infty$, we have
\[ W_p (\nu_N, \mathrm{Vol}) \ll p^2 N^{-1/2} \]
with an implied constant that depends only on $c$.
\end{lem}

\begin{proof} We may assume $p+2 \le \lfloor \sqrt{cN} \rfloor$ by Lemma \ref{WqWplemma}. Let $1<q \le 2$ be such that $1/p+1/q=1$, and let us apply Theorem \ref{torussmoothing_theorem1} with $H=\lfloor \sqrt{cN} \rfloor$ to deduce
\[ W_p (\nu_N, \mathrm{Vol}) \ll \frac{p}{\sqrt{N}} + p \bigg( \sum_{\substack{k \in [-H,H]^2 \\ k \neq 0}} \frac{|\widehat{\nu_N(k)}|^q}{|k|^q} \bigg)^{1/q} . \]
The Fourier transform of $\nu_N(k)$ for $k=(k_1, k_2)$, $k_1 \neq 0$, can be estimated as
\[ \begin{split} |\widehat{\nu_N}(k)| &= \frac{1}{2N} \bigg| \sum_{n=0}^{N-1} \left( e^{-2 \pi i (k_1 \alpha + k_2/(2N))n} + e^{-2 \pi i (k_1 \alpha - k_2 /(2N))n} \right) \bigg| \\ &\le \frac{1}{N} \left( \frac{1}{|1-e^{-2 \pi i (k_1 \alpha + k_2 /(2N))}|} +\frac{1}{|1-e^{-2 \pi i (k_1 \alpha - k_2 /(2N))}|} \right) \\ &\ll \frac{1}{N} \left( \frac{1}{\| k_1 \alpha + k_2/(2N) \|_{\mathbb{R}/\mathbb{Z}}} + \frac{1}{\| k_1 \alpha - k_2/(2N) \|_{\mathbb{R}/\mathbb{Z}}} \right) . \end{split} \]
For all $k \in [-H,H]^2$, $k_1 \neq 0$, we have $|k_1|, |k_2| \le H \le \sqrt{cN}$, therefore $\| k_1 \alpha \|_{\mathbb{R}/\mathbb{Z}} \ge c/|k_1| \ge |k_2|/N$ and
\[ \| k_1 \alpha \pm k_2/(2N) \|_{\mathbb{R}/\mathbb{Z}} \ge \| k_1 \alpha \|_{\mathbb{R}/\mathbb{Z}} - \frac{|k_2|}{2N} \ge \frac{1}{2} \| k_1 \alpha \|_{\mathbb{R}/\mathbb{Z}} . \]
Consequently, $|\widehat{\nu_N}(k)| \ll 1/(N \| k_1 \alpha \|_{\mathbb{R}/\mathbb{Z}})$ whenever $k_1 \neq 0$. If $k \in [-H,H]^2$, $k_1 =0$, then
\[ \begin{split} \widehat{\nu_N}(k) &= \frac{1}{2N} \sum_{n=0}^{N-1} \left( e^{-2 \pi i k_2 n/(2N)} + e^{2 \pi i k_2 n /(2N)} \right) = \frac{1}{2N} \left( \frac{1-(-1)^{k_2}}{1-e^{- 2 \pi i k_2 /(2N)}} + \frac{1-(-1)^{k_2}}{1-e^{2 \pi i k_2/(2N)}} \right) \\ &= \left\{ \begin{array}{ll} 0 & \textrm{if } k_2 \textrm{ is even}, \\ 1/N & \textrm{if } k_2 \textrm{ is odd}. \end{array} \right. \end{split} \]
Hence
\[ W_p (\nu_N, \mathrm{Vol}) \ll \frac{p}{\sqrt{N}} + p \bigg( \sum_{\substack{k \in [-H,H]^2 \\ k_1 \neq 0}} \frac{1}{N^q |k|^q \| k_1 \alpha \|_{\mathbb{R}/\mathbb{Z}}^q} + \sum_{\substack{k_2 \in [-H,H] \\ k_2 \neq 0}} \frac{1}{N^q |k_2|^q} \bigg)^{1/q} . \]
The second sum in the previous formula can be estimated as
\[ \sum_{\substack{k_2 \in [-H,H] \\ k_2 \neq 0}} \frac{1}{N^q |k_2|^q} \le \frac{2}{N^q} \sum_{n=1}^{\infty} \frac{1}{n^q} \ll \frac{1}{(q-1)N^q} \ll \frac{p}{N^q}. \]
Let $1 \le Q \le H$. In particular, $2Q \le 2\sqrt{cN} \le \sqrt{N}$. For any $k_1, k_1' \in [-Q,Q]$, $k_1 \neq k_1'$, we have $\| k_1 \alpha - k_1' \alpha \|_{\mathbb{R}/\mathbb{Z}} \ge c/|k_1-k_1'| \gg 1/Q$. By the pigeonhole principle, we thus have
\[ \sum_{0<|k_1| \le Q} \frac{1}{\| k_1 \alpha \|_{\mathbb{R}/\mathbb{Z}}^q} \ll \sum_{n=1}^{\infty} \frac{1}{(n/Q)^q} \ll p Q^q. \]
Therefore
\[ \sum_{Q/2 < \| k \|_{\infty} \le Q} \frac{1}{|k|^q \| k_1 \alpha \|_{\mathbb{R}/\mathbb{Z}}^q} \ll \frac{1}{Q^q} \sum_{|k_2| \le Q} \sum_{0<|k_1| \le Q} \frac{1}{\| k_1 \alpha \|_{\mathbb{R}/\mathbb{Z}}^q} \ll p Q, \]
and a dyadic decomposition of the interval $[1,H]$ leads to
\[ \sum_{\substack{k \in [-H,H]^2 \\ k_1 \neq 0}} \frac{1}{|k|^q \| k_1 \alpha \|_{\mathbb{R}/\mathbb{Z}}^q} \ll p H \ll p \sqrt{N} . \]
Hence
\[ W_p (\nu_N, \mathrm{Vol}) \ll \frac{p}{\sqrt{N}} + p \bigg( \frac{p \sqrt{N}}{N^q} \bigg)^{1/q} . \]
The exponent of $N$ in the second term is $1/(2q) -1 < -1/2$, and the claim follows.
\end{proof}

Since $f$ is Lipschitz with Lipschitz constant 2 and preserves the Lebesgue measure $\mathrm{Vol}$, an application of Lemma \ref{lipschitz_lemma} gives
\begin{equation}\label{fpushforward}
W_p (f_* \nu_N, \mathrm{Vol}) \le 2 W_p (\nu_N, \mathrm{Vol}) .
\end{equation}
In contrast, the Lambert projection $\phi: (\mathbb{R}/\mathbb{Z}) \times [0,1] \to \mathbb{S}^2$ is not Lipschitz, so Lemma \ref{lipschitz_lemma} does not apply. We can nevertheless estimate the effect of the pushforward with respect to $\phi$ as follows.

\begin{lem}\label{phipushforward_lemma} Let $\mu \in \mathcal{P}((\mathbb{R}/\mathbb{Z}) \times [0,1])$ and $1 \le p < 2$. For any $p'>2p/(2-p)$, we have
\[ W_p (\phi_* \mu, \mathrm{Vol}) \ll \left( 1-\frac{pp'}{2(p'-p)} \right)^{1/p'-1/p} W_{p'}(\mu, \mathrm{Vol}) \]
with a universal implied constant.
\end{lem}

\begin{proof} For the sake of readability, let $M=(\mathbb{R}/\mathbb{Z}) \times [0,1]$ denote the cylinder. Let $\pi \in \mathrm{Coupling}(\mu, \mathrm{Vol})$ for which the minimum in the definition of the Wasserstein metric $W_{p'}(\mu, \mathrm{Vol})$ is attained:
\[ W_{p'} (\mu, \mathrm{Vol}) = \left( \int_{M \times M} \rho_M (x,y)^{p'} \, \mathrm{d}\pi (x,y) \right)^{1/p'} . \]
Let $F: M \times M \to \mathbb{S}^2 \times \mathbb{S}^2$, $F(x,y)=(\phi(x), \phi(y))$. One readily checks that the pushforward $F_* \pi \in \mathcal{P}(\mathbb{S}^2 \times \mathbb{S}^2)$ is a coupling between $\phi_* \mu$ and $\phi_* \mathrm{Vol}=\mathrm{Vol}$. By the definition of the Wasserstein metric,
\[ \begin{split} W_p (\phi_* \mu, \mathrm{Vol}) &\le \left( \int_{\mathbb{S}^2 \times \mathbb{S}^2} \rho_{\mathbb{S}^2} (u,v)^p \, \mathrm{d} (F_* \pi)(u,v) \right)^{1/p} \\ &= \left( \int_{M \times M} \rho_{\mathbb{S}^2} (\phi(x), \phi(y))^p \, \mathrm{d}\pi (x,y) \right)^{1/p} . \end{split} \]
The partial derivatives of the Lambert map are
\[ \begin{split} \frac{\partial \phi}{\partial x_1} &= \left( -4 \pi \sqrt{x_2-x_2^2} \sin (2 \pi x_1), 4 \pi \sqrt{x_2-x_2^2} \cos (2 \pi x_1), 0 \right), \\ \frac{\partial \phi}{\partial x_2} &= \left( \frac{1-2x_2}{\sqrt{x_2-x_2^2}} \cos (2 \pi x_1), \frac{1-2x_2}{\sqrt{x_2-x_2^2}} \sin (2 \pi x_1), -2 \right) . \end{split} \]
In particular,
\[ \left| \frac{\partial \phi}{\partial x_1} \right| \ll 1 \qquad \textrm{and} \qquad \left| \frac{\partial \phi}{\partial x_2} \right| \ll \frac{1}{\sqrt{x_2-x_2^2}} \ll \frac{1}{\sqrt{x_2}} + \frac{1}{\sqrt{1-x_2}} . \]
For all $x=(x_1, x_2), y=(y_1, y_2) \in M$, we thus deduce
\[ \begin{split} |\phi (x) - \phi(y)| &\le |\phi (x_1, x_2) - \phi (y_1, x_2)| + |\phi (y_1, x_2) - \phi (y_1, y_2)| \\ &\ll \| x_1-y_1 \|_{\mathbb{R}/\mathbb{Z}} + \left| \int_{x_2}^{y_2} \left( \frac{1}{\sqrt{t}} + \frac{1}{\sqrt{1-t}} \right) \, \mathrm{d}t \right| \\ &\ll \| x_1-y_1 \|_{\mathbb{R}/\mathbb{Z}} + \frac{|x_2-y_2|}{\sqrt{x_2} + \sqrt{y_2}} + \frac{|x_2-y_2|}{\sqrt{1-x_2} + \sqrt{1-y_2}} \\ &\ll \rho_M (x,y) \left( \frac{1}{\sqrt{y_2}} + \frac{1}{\sqrt{1-y_2}} \right) . \end{split} \]
Since the geodesic metric $\rho_{\mathbb{S}^2}$ is bounded above by a constant times the Euclidean metric, we have
\[ \rho_{\mathbb{S}^2} (\phi(x), \phi(y))^p \ll \rho_M (x,y)^p \left( \frac{1}{y_2^{p/2}} + \frac{1}{(1-y_2)^{p/2}} \right) . \]
An application of the H\"older inequality with exponents $p'/p>1$ and $p'/(p'-p)$ leads to
\[ \begin{split} W_p (\phi_* \mu, \mathrm{Vol}) &\ll \left( \int_{M \times M} \rho_M (x,y)^p \left( \frac{1}{y_2^{p/2}} + \frac{1}{(1-y_2)^{p/2}} \right) \, \mathrm{d}\pi (x,y) \right)^{1/p} \\ &\ll \left( \int_{M \times M} \rho_M (x,y)^{p'} \, \mathrm{d}\pi (x,y) \right)^{1/p'} \\ &{\phantom{\ll}} \times \left( \int_{M \times M} \left( \frac{1}{y_2^{pp'/(2(p'-p))}} + \frac{1}{(1-y_2)^{pp'/(2(p'-p))}} \right) \, \mathrm{d}\pi (x,y) \right)^{(p'-p)/(pp')} . \end{split} \]
The first integral in the previous formula is $W_{p'}(\mu, \mathrm{Vol})$ because $\pi$ was chosen to be an optimal coupling. Since the second integral is a function of $y$, $\pi$ is a coupling between $\mu$ and $\mathrm{Vol}$, and $pp'/(2(p'-p))<1$ by assumption, we obtain the desired estimate
\[ \begin{split} W_p (\phi_* \mu, \mathrm{Vol}) &\ll W_{p'}(\mu, \mathrm{Vol}) \left( \int_M \left( \frac{1}{y_2^{pp'/(2(p'-p))}} + \frac{1}{(1-y_2)^{pp'/(2(p'-p))}} \right) \, \mathrm{d}\mathrm{Vol}(y) \right)^{(p'-p)/(pp')} \\ &\ll W_{p'}(\mu, \mathrm{Vol}) \left( \int_0^{1/2} \frac{1}{t^{pp'/(2(p'-p))}} \, \mathrm{d}t \right)^{1/p-1/p'} \\ &\ll W_{p'}(\mu, \mathrm{Vol}) \left( 1-\frac{pp'}{2(p'-p)} \right)^{1/p'-1/p} . \end{split} \]
\end{proof}

\begin{proof}[Proof of Theorem \ref{fibonacci_theorem}] Fix $1 \le p < 2$, and let $p'>2p/(2-p)$ to be chosen. An application of Lemmas \ref{fibonacci_lemma} and \ref{phipushforward_lemma} and formula \eqref{fpushforward} lead to
\[ \begin{split} W_p (\mu_N, \mathrm{Vol}) &= W_p (\phi_* f_* \nu_N, \mathrm{Vol}) \ll \left( 1-\frac{pp'}{2(p'-p)} \right)^{1/p'-1/p} W_{p'}(f_* \nu_N, \mathrm{Vol}) \\ &\ll \left( 1-\frac{pp'}{2(p'-p)} \right)^{1/p'-1/p} W_{p'}(\nu_N, \mathrm{Vol}) \ll \left( 1-\frac{pp'}{2(p'-p)} \right)^{1/p'-1/p} (p')^2 N^{-1/2} . \end{split} \]
Choosing $p'=4p/(2-p)$ proves the claim $W_p (\mu_N, \mathrm{Vol}) \ll (2-p)^{-5/2} N^{-1/2}$.
\end{proof}

\vspace{5mm}

\noindent\textsc{Bence Borda}

\noindent\textsc{Department of Mathematics, University of Sussex, Brighton, BN1 9RH, United Kingdom}

\noindent\textsc{Email:} \texttt{b.borda@sussex.ac.uk}

\vspace{5mm}

\noindent\textsc{Filippo Giannoni}

\noindent\textsc{Department of Mathematics, University of Sussex, Brighton, BN1 9RH, United Kingdom}

\noindent\textsc{Email:} \texttt{f.giannoni@sussex.ac.uk}

\end{document}